\documentclass[pdf, 12pt]{article}
 \usepackage{pdftricks,pst-node}
\usepackage{graphicx,amsmath,amssymb,amsthm}
\usepackage{marsden_article}
\usepackage{epstopdf}
\usepackage{cite}

\input xy
\xyoption{all}

\def\E{\mathbb{E}}

\newtheorem{Assumption}{Assumption}[section]
\newtheorem{Remark}{Remark}[section]

\newtheorem*{thm21}{Theorem 2.1}
\newtheorem*{thm22}{Theorem 2.2}
\newtheorem*{thm23}{Theorem 2.3}
\begin{document}

\title{
Pathwise Accuracy \& Ergodicity of \\
Metropolized Integrators for SDEs
}

\author{Nawaf Bou-Rabee\thanks{Courant Institute of Mathematical Sciences, New York University, 
251 Mercer Street, New York, NY 10012-1185 ({\tt nawaf@cims.nyu.edu}).
 N.~B-R.~was supported by NSF Fellowship \# DMS-0803095. }  \and 
Eric Vanden-Eijnden\thanks{Courant Institute of Mathematical Sciences, New York University, 
251 Mercer Street, New York, NY 10012-1185 ({\tt eve2@cims.nyu.edu}).} }
\maketitle

\medskip

\begin{abstract}

	Metropolized integrators for ergodic
	stochastic differential equations (SDE) are proposed which (i) are
	ergodic with respect to the (known) equilibrium distribution of the
	SDE and (ii) approximate pathwise the solutions of the SDE on
	finite time intervals. Both these properties are demonstrated in the
	paper and precise strong error estimates are obtained. It is also
	shown that the Metropolized integrator retains these properties 
	even in situations where the drift in the SDE is nonglobally
	Lipschitz, and vanilla explicit integrators for SDEs typically
	become unstable and fail to be ergodic.
  
\end{abstract}


\medskip




\section{Introduction}

\paragraph{Purpose of the Paper.}

This paper considers SDEs whose solution possesses a probability
transition density that satisfies a detailed balance-type condition with respect
to a known probability distribution.  In this context the paper analyzes 
integrators for such SDEs that
\begin{description}
\item[(i)]\label{ergo} are ergodic with respect to the exact equilibrium
  distribution of the SDE on infinite time intervals; and,
\item[(ii)]\label{strong} strongly converge to the solutions of the SDE on
  finite time intervals.
\end{description}
Pure sampling methods can accomplish (i), but they typically
do not approximate the solution to the SDE.  Integrators for SDEs
certainly satisfy (ii), but they often are divergent on
infinite time intervals or ergodic with respect to a different
equilibrium distribution \cite{TaTu1990, Ta1995, MaStHi2002,
  MiTr2004}.  This paper shows that a Metropolized integrator can
simultaneously accomplish these goals.

\paragraph{Motivation.}

This paper is motivated by SDEs that arise in molecular dynamics (MD).
Although the numerical analysis of general SDEs is well-studied (see,
e.g., \cite{Ta1995, MiTr2004}), the relevance of this analysis to SDEs
that arise in MD is limited.  Indeed such SDEs are characterized by
complex drift vector fields that have limited regularity, and one is typically
interested in their solutions over long time intervals. Let us
elaborate briefly on the consequences of these two observations.

The drift vector field of SDEs that arise in MD involve the
gradient of an empirically derived potential energy function.  Due
to singularities in the potential energy, this drift vector field
possesses limited regularity, e.g., it is nonglobally Lipschitz.  
The drift vector field is also computationally costly to evaluate 
because the potential force involves intricate short and 
long-range interactions between atoms.  As a result implicit 
methods are typically too costly in this application area.

Moreover, the lack of regularity in the drift of SDEs that arise in MD
causes explicit discretizations to be stochastically unstable in
general.  This implies that, in principle, such schemes cannot be used 
to sample the equilibrium distribution of the SDE. Some approaches to
stabilize explicit discretizations of SDEs include
\begin{itemize}
\item adaptive time-stepping~\cite{LaMaSt2007};
\item multiple time-stepping~\cite{TuBe1991};
\item method of rejecting exploding trajectories~\cite{MiTr2005}; and
\item using bounded random numbers instead of Gaussian ones.
\end{itemize}
These methods are often ergodic but typically introduce discretization
errors in the sampling (i.e. the equilibrium distribution of the
numerical scheme is not exactly the same as the distribution of the
SDE). A standard way to remove these errors is to resort to a Monte-Carlo
method designed for sampling, e.g., hybrid Monte-Carlo algorithms 
\cite{DuKePeRo1987, Ho1991, AkRe2008}.  This procedure, however, has the 
effect that the dynamics of the method becomes unrelated to that of the SDE. 
Indeed, in this context, the main question often becomes how to make the 
sampling as efficient as possible by accelerating the rate of convergence 
of the method.

Here we show that a Metropolis-adjusted explicit integrator can
both be ergodic with respect to the known equilibrium distribution of
the SDE and captures the dynamical behavior of the solutions of this
SDE. To establish this second property, one needs to estimate the
effect on the dynamics of the rejections in the Metropolis-Hasting
method.  The analysis of this effect in the context of a
Metropolized integrator for SDEs with the special structure relevant to 
MD is one of the main objectives of this paper.

\paragraph{Organization of the Paper.}

In \S \ref{sec:MainResults}, the main results of the paper are presented 
without proofs and put into context.  In \S \ref{sec:NumericalIllustration}, 
numerical validation of the main results is provided.  The proofs of the 
main results are organized according to the type of dynamic and 
discretization considered.  In  \S \ref{sec:OverdampedLangevin},  the proofs 
pertaining to a forward Euler-Maruyama discretization of 
overdamped Langevin dynamics and its Metropolis-adjustment can be found.  
In \S \ref{sec:InertialLangevin},  the proofs pertaining to a 
St\"{o}rmer-Verlet based discretization  of inertial 
Langevin dynamics and its Metropolis-adjustment can be found.  
We wrap up the paper with a conclusion in \S \ref{sec:Conclusion}.


\section{Main Results} \label{sec:MainResults}


\subsection{Overdamped Langevin}

In the first part of the paper, we shall focus on overdamped Langevin 
dynamics on a energy landscape defined by a potential energy 
function $U \in C^{\infty}(\mathbb{R}^n, \mathbb{R})$: 
\begin{equation}
    d\boldsymbol{Y} = -\nabla U(\boldsymbol{Y}) dt + \sqrt{2 \beta^{-1} } d \boldsymbol{W}
    \label{SDE1}
\end{equation}
Here $\nabla U(\boldsymbol{x}): \mathbb{R}^n \to \mathbb{R}^n$ denotes the
gradient of the function $U$, $\boldsymbol{W}$ is a standard
$n$-dimensional Wiener process, or Brownian motion, and $\beta>0$ is a
parameter referred to as the inverse temperature.  Under certain
regularity conditions on the potential energy, the solution to \eqref{SDE1} is
geometrically ergodic with an invariant probability measure $\mu$ that
possesses the following density $\pi(\boldsymbol{x})$ with respect to 
Lebesgue measure\cite{Ha1980, RoTw1996A, RoTw1996B}:
\begin{equation}
  \pi(\boldsymbol{x}) = Z^{-1} \exp(-\beta U(\boldsymbol{x}))
  \label{IM}
\end{equation}
where $Z= \int_{\mathbb{R}^n} \exp(-\beta U(\boldsymbol{x}))d\boldsymbol{x}$. 
Next we recall various integration strategies for \eqref{SDE1} and 
summarize their properties.

\paragraph{Forward Euler-Maruyama.}

Let $N$ and $h$ be given, set $T= N h$ and $t_k = h k $ for
$k=0,...,N$, and consider the following explicit Euler-Maruyama
discretization of~\eqref{SDE1}:
\begin{equation}
    \Tilde{\boldsymbol{X}}_{k+1} =\Tilde{\boldsymbol{X}}_k 
    - h \nabla U(\Tilde{\boldsymbol{X}}_k)
    +  \sqrt{2 \beta^{-1}}  (\boldsymbol{W}(t_{k+1}) - \boldsymbol{W}(t_{k}))
    \label{ULA}
\end{equation}
where $ \Tilde{\boldsymbol{X}}_k \approx \boldsymbol{Y}(t_k)$.
The iteration rule~\eqref{ULA} defines a Markov chain that possesses a
transition kernel with the following smooth, strictly positive probability 
transition density:
\begin{equation}
  q_h(\boldsymbol{x}, \boldsymbol{y}) = (4 \pi \beta^{-1} h)^{-n/2} 
  \exp\left(  -  \frac{\left| \boldsymbol{y} - \boldsymbol{x} 
        + h \nabla U(\boldsymbol{x} ) \right|^2}{ 4 \beta^{-1} h} \right) 
  \label{ULAdensity}
\end{equation}  
Hence, the chain is irreducible with respect to Lebesgue
measure.  To be consistent with the literature, the discrete-time 
Markov chain generated by Euler-Maruyama is called the 
unadjusted Langevin algorithm (ULA)\cite{RoTw1996A}.

If $\nabla U$ is globally Lipschitz and $h$ is small
enough, ULA~\eqref{ULA} can often be shown to be:
\begin{itemize}
\item first-order strongly convergent to the solution orbits of \eqref{SDE1} on finite time
intervals;
\item geometrically ergodic on infinite-time intervals with respect to an invariant measure 
that is a first-order approximant to the invariant measure $\mu$ of \eqref{SDE1}.
\end{itemize}
The second property is typically established using a Talay-Tubaro expansion of the 
global weak error of  ULA\cite{TaTu1990}.

On the other hand, when $\nabla U$ is nonglobally Lipschitz ULA becomes a 
transient Markov chain for any $h>0$\cite{Nu1984,MeTw1996}.  
In fact, all moments of Euler-Maruyama 
are unbounded on long time-intervals for any initial condition 
$\boldsymbol{x} \in \mathbb{R}^n$, i.e., for any
integer $\ell \ge 1$ and for any $h>0$
\begin{equation}
  \E^{\boldsymbol{x}} \{ | \Tilde{\boldsymbol{X}}_k |^{2 \ell} \} \to \infty ~~~\text{as}~~~k \to \infty 
\label{emunboundedmoments}
\end{equation}
where $\mathbb{E}^{\boldsymbol{x}}$ denotes the expectation 
conditional on $\Tilde{\boldsymbol{X}}_0=\boldsymbol{x}$. 
See, e.g.,\cite{MaStHi2002, Ta2002}.   As is well known in the literature, 
a Metropolis-Hastings method can stabilize ULA.

\paragraph{Metropolized Forward Euler-Maruyama.}

A Metropolis-Hastings method is a quite general Monte-Carlo method for
sampling from a known probability distribution \cite{MeRoRoTeTe1953,Ha1970}.  
The method generates a Markov chain from a given proposal Markov chain 
as follows.  The algorithm computes a proposal move according to the 
proposal chain and accepts this proposal with a probability that ensures the
Metropolized chain is ergodic with respect to the given probability
distribution.  Here we shall focus on the Metropolized Euler-Maruyama
integrator defined in terms of the equilibrium density $\pi$~\eqref{IM}
and the transition density $q_h$~\eqref{ULAdensity}.

Let $\zeta_k \sim U(0,1)$ for $k=0,...,N-1$.  Given $h$ and
$\boldsymbol{X}_k$ the algorithm calculates a proposal move
using the Euler-Maruyama updating scheme in~\eqref{ULA}:
\begin{equation}   \label{MALAproposal}
  \boldsymbol{X}^*_{k+1} = \boldsymbol{X}_k - h
  \nabla U(\boldsymbol{X}_k ) + \sqrt{2 \beta^{-1}}
  (\boldsymbol{W}(t_{k+1}) - \boldsymbol{W}(t_{k}))
\end{equation}
and accepts this proposal with a probability 
\begin{equation}
  \label{MALAacceptreject}
  \alpha_h(\boldsymbol{x}, \boldsymbol{y}) =  1 \wedge 
  \frac{q_h(\boldsymbol{y},\boldsymbol{x}) \pi(\boldsymbol{y}) }
  { q_h(\boldsymbol{x},\boldsymbol{y}) \pi(\boldsymbol{x}) }  \text{.}
\end{equation}
That is, the update is defined as:
\begin{equation}
  \label{MALA}
  \boldsymbol{X}_{k+1} =
\begin{cases}
  \boldsymbol{X}^*_{k+1}\qquad &
  \text{if}~~\zeta_k<\alpha_h(\boldsymbol{X}_k,
  \boldsymbol{X}^*_{k+1} )\\
  \boldsymbol{X}_k & \text{otherwise}
\end{cases} 
\end{equation}
for $k=0,...,N-1$.  To be consistent with the literature, we will refer 
to the Metropolized Euler-Maruyama integrator as the Metropolis-adjusted 
Langevin algorithm (MALA)\cite{RoTw1996A}. By construction, MALA preserves 
the invariant measure $\mu$ of \eqref{SDE1}.   If $\nabla U$ is globally Lipschitz and $h$ 
small enough, one can often show that MALA is geometrically ergodic\cite{RoTw1996A}.  
See, e.g., Theorem 4.1 of\cite{RoTw1996A}.   However,  if $\nabla U$ is nonglobally Lipschitz, 
MALA is often no longer geometrically ergodic.   Still, 
for any  $g: \mathbb{R}^n \to \mathbb{R}$ with $\mu(g) < \infty$,
\begin{equation} \label{Xpreservesmu}
\E_{\mu} \E^{\boldsymbol{x}} \left\{ g( \boldsymbol{X}_k ) \right\}   
= \int_{\mathbb{R}^n} g d\mu,~~\forall~~ k \in \mathbb{N}
\end{equation}
where $\E_{\mu} \E^{\boldsymbol{x}}$ denotes
expectation conditioned on the initial distribution being the
equilibrium distribution of \eqref{SDE1}, i.e.,
\[
\mathbb{E}_{\mu} \mathbb{E}^{\boldsymbol{x}} \left\{ g( \boldsymbol{X}_k ) \right\} = 
\int_{\mathbb{R}^n} \mathbb{E}^{\boldsymbol{x}} \left\{ g( \boldsymbol{X}_k ) \right\}  \mu(d \boldsymbol{x}) \text{.}
\]
The identity \eqref{Xpreservesmu} is obviously a significant 
improvement to \eqref{emunboundedmoments}.

When $\nabla U$ is nonglobally Lipschitz, MALA is not geometrically ergodic because 
ULA is transient (see Theorem 4.2 of\cite{RoTw1996A}).  To correct this problem, 
Roberts and Tweedie proposed a modification of MALA which truncates the drift in regions 
where the underlying Euler method can be explosive.   Specifically the proposal move
\eqref{MALAproposal} is modified to a MALA algorithm with bounded drift:
\begin{align}  \label{MALTAproposal}
  \boldsymbol{Z}^*_{k+1} =   \boldsymbol{Z}_k 
  -  h \frac{\nabla U( \boldsymbol{Z}_k )}{1 \vee h | \nabla U( \boldsymbol{Z}_k ) | }  
  + \sqrt{2 \beta^{-1}} (\boldsymbol{W}(t_{k+1}) - \boldsymbol{W}(t_{k}))  
\end{align}
When $ | \nabla U( \boldsymbol{Z}_k ) | < 1/h$, \eqref{MALTAproposal} 
is the Euler-Maruyama update, and otherwise, the proposal drift
in \eqref{MALTAproposal} preserves the direction of the drift in ULA, 
but normalizes the amplitude in all degrees of freedom.   The Metropolis-Hastings 
method with this modified proposal move will be referred to  as the 
Metropolis-adjusted Langevin truncated algorithm (MALTA)\cite{RoTw1996A}.  
By comparison to a random-walk-based Metropolis algorithm, one can often show that
MALTA is geometrically ergodic\cite{RoTw1996A, At2005, JaHa2000}.   
While this {\em ad hoc} correction to MALA makes MALTA geometrically 
ergodic, the relation of MALTA to the original diffusion \eqref{SDE1} 
remains to be established.

\paragraph{Main Result I: Strong Convergence of MALA \& MALTA.}

In the paper we shall consider the situation where $\nabla U$ is nonglobally 
Lipschitz. Roughly speaking, MALTA corrupts Euler-Maruyama orbit 
by the random rejections in~\eqref{MALA} with the modified 
proposal~\eqref{MALTAproposal} and the {\em ad hoc} truncation of the drift.  
Yet, we show in the paper that MALTA still approximates pathwise the solution 
of \eqref{SDE1}.  The precise statement is:
\begin{theorem}[MALTA Strong Accuracy]
  Assume \ref{sa}.    Then for every $E_0>0$ and $T>0$,
  there exists $h_c(E_0) > 0$ and $C(T, E_0)>0$ such that for all
  $h<h_c$, for all $\boldsymbol{x} : U(\boldsymbol{x}) \le E_0$, and
  for all $t \in [0, T]$,
\[
\left( \E^{\boldsymbol{x}} \left\{ \left|\boldsymbol{Z}_{\lfloor t/h
        \rfloor} - \boldsymbol{Y}(t) \right|^2 \right\}
\right)^{1/2} \le C(T, E_0) h^{3/4} \text{.}
\]
\label{MALTAaccuracy}
\end{theorem}
Assumption~\ref{sa} can be found in \S\ref{preliminaries}.
This assumption remains valid for SDEs of the type \eqref{SDE1} in which 
the drift is nonglobally Lipschitz.  A key ingredient in the proof is 
geometric ergodicity of MALTA.   This property implies bounds on moments 
of MALTA at finite times.   The precise bounds can be found in Lemma 
\ref{MALTAmomentbound}.

Since MALA is not geometrically ergodic when $\nabla U$ is nonglobally Lipschitz, 
we are unable to obtain such bounds on moments of MALA at finite times.   
However, such bounds can be derived if MALA's initial condition is restricted to
the equilibrium distribution of \eqref{SDE1}.  This is because MALA by design 
preserves this distribution.  Hence, for MALA we are able to prove the following theorem.
We stress that MALTA's strong accuracy does not require this restriction on initial conditions.
\begin{theorem}[MALA Strong Accuracy from Equilibrium]
Assume~\ref{sa}. For all $T>0$ there exists $h_c > 0$ and $C(T)>0$ 
such that for all positive $h<h_c$ and for all $t \in [0, T]$,
\[
\left( \E_{\mu}\E^{\boldsymbol{x}} \left\{ \left|
      \boldsymbol{X}_{\lfloor t/h \rfloor} - \boldsymbol{Y}(t)
    \right|^2 \right\} \right)^{1/2} \le C(T) h^{3/4} \text{.}
\]
\label{MALAaccuracy}
\end{theorem}
The proofs of Theorems~\ref{MALTAaccuracy} and~\ref{MALAaccuracy}  
as well as the precise  statement of Assumption~\ref{sa} on which they rely 
are presented in \S\ref{sec:OverdampedLangevin}.  The proofs build on results 
of global accuracy for numerical methods applied to SDEs with  uniformly 
Lipschitz drift\cite{MiTr2004} and with one-sided Lipschitz drift\cite{HiMaSt2002}.  
\S\ref{sec:NumericalIllustration} will illustrate these results on a simple but
representative example with a nonglobally Lipschitz drift.


\subsection{Inertial Langevin}

Next, we shall consider Langevin equations 
defined in terms of a Hamiltonian function $H \in C^{\infty}(\mathbb{R}^{2n}, \mathbb{R})$:
\begin{equation} 
d \mathbf{Y} = \mathbb{J} \nabla H(  \mathbf{Y} ) dt 
- \gamma \boldsymbol{C} \nabla H(  \mathbf{Y} ) dt 
+ \sqrt{2 \gamma \beta^{-1}} \boldsymbol{C}  d \mathbf{W}
\label{SDE2}
\end{equation}
where the following matrices have been introduced:
\[
\mathbb{J}= 
\begin{bmatrix}
0 & \mathbf{I} \\
-\mathbf{I} &  0 
\end{bmatrix},~~~
\boldsymbol{C} = 
\begin{bmatrix}
0 & 0 \\
0 & \mathbf{I} 
\end{bmatrix}  \text{.}
\]
Here $\boldsymbol{W}$ is a standard $2n$-dimensional Wiener process, 
or Brownian motion, $\beta>0$ is a parameter referred to as the inverse 
temperature, and $\gamma>0$ is referred to as the friction coefficient.  Write 
$\mathbf{Y}(t)=(\boldsymbol{Q}(t),\boldsymbol{P}(t))$ where $\boldsymbol{Q}(t)$ 
and  $\boldsymbol{P}(t)$ represent the instantaneous configuration and 
momentum of the system, respectively.   We shall assume:
\[
H(\boldsymbol{q}, \boldsymbol{p}) = 
\frac{1}{2}  \boldsymbol{p}^T \boldsymbol{M}^{-1} \boldsymbol{p} + U(\boldsymbol{q}) \text{,}
\]
where $\boldsymbol{M}$ is a symmetric positive definite mass matrix and 
$U$ is a potential energy function.   There are two main differences between
the overdamped and inertial Langevin dynamics.  First, the solution
of \eqref{SDE1} is a reversible stochastic process, whereas the solution of 
\eqref{SDE2} is not.  Second, the diffusion in \eqref{SDE2} is only applied to 
momentum degrees of freedom, whereas the diffusion in \eqref{SDE1} is 
applied to all degrees of freedom.  The consequences of these differences
to the discretization of \eqref{SDE2} and its Metropolis-adjustment will be a 
main focus of \S \ref{sec:InertialLangevin}.

Despite the degenerate diffusion in \eqref{SDE2}, 
under certain regularity conditions on $U$, the solution to this SDE is geometrically 
ergodic with respect to an invariant probability measure $\mu$ with the following 
density\cite{Ta2002}:
\begin{equation}
\pi(\boldsymbol{q}, \boldsymbol{p}) = 
Z^{-1} \exp\left( -  \beta H(\boldsymbol{q}, \boldsymbol{p}) \right) \text{,}
  \label{IM2}
\end{equation}
where $Z= 
\int_{\mathbb{R}^{2n}} 
\exp\left(-   \beta H(\boldsymbol{q}, \boldsymbol{p}) \right) 
d\boldsymbol{q} d\boldsymbol{p}$.

\paragraph{Geometric Langevin Algorithm.}

Let $N$ and $h$ be given, set $T=  N h$ and $t_k = h k $ for
$k=0,...,N$.   We shall consider an integrator for \eqref{SDE2} based on 
splitting the Langevin equations into Hamilton's equations for the Hamiltonian
$H$:
\begin{equation} \label{HamiltonsEquations}
\begin{cases} 
d \boldsymbol{Q} &= \boldsymbol{M}^{-1} \boldsymbol{P} dt,  \\
d \boldsymbol{P}  &= - \nabla U(\boldsymbol{Q}) dt, 
\end{cases}
\end{equation}
 and Ornstein-Uhlenbeck equations
\begin{equation} \label{OrnsteinUhlenbeck}
\begin{cases} 
d \boldsymbol{Q} &= 0, \\
d \boldsymbol{P}  &= - \gamma \boldsymbol{M}^{-1} \boldsymbol{P} dt 
				 + \sqrt{2 \beta^{-1} \gamma} d \boldsymbol{W}  \text{.}
\end{cases}
\end{equation}
The solution of Hamilton's equations will be approximated by a symplectic integrator; to be specific, 
the discrete Hamiltonian map of a self-adjoint variational integrator\cite{MaWe2001}.   While the 
exact flow will be used for the Ornstein-Uhlenbeck equations.  These flows will be 
composed in a Strang-type splitting to obtain a pathwise approximant to the solution 
of inertial Langevin which we will refer to as the Geometric Langevin Algorithm (GLA).

This type of splitting of inertial Langevin equations is quite natural and 
has been recently used in simulations of molecular dynamics (see \cite{VaCi2006, BuPa2007}),  
dissipative particle dynamics  (see \cite{Sh2003, SeFaEsCo2006}), and
inertial particles (see \cite{PaStZy2008}).   This paper is geared 
 towards applications in molecular dynamics where Langevin integrators 
 (including the ones cited above) have been based on generalizations of the widely used 
 St\"{o}rmer-Verlet integrator.   The St\"{o}rmer-Verlet integrator is attractive for molecular dynamics 
 because it is an explicit, symmetric, second-order accurate, variational integrator for Hamilton's equations. 
 In molecular dynamics it was popularized by Loup Verlet in 1967.   Other popular generalizations 
 of the St\"{o}rmer-Verlet integrator to Langevin equations include 
 Br\"{u}nger-Brooks-Karplus (BBK) \cite{BrBrKa1984}, 
 van Gunsteren and Berendsen (vGB) \cite{GuBe1982}, 
 and the Langevin-Impulse (LI) methods  \cite{SkIz2002}.  
 The LI method is also based on a splitting of Langevin equations, 
 but it is different from the splitting considered here.     The long-time properties
 of GLA were recently analyzed in the context of uniformly Lipschitz potential forces
 in \cite{BoOw2009}.

An example of a GLA that will be analyzed in \S \ref{sec:InertialLangevin} is 
the St\"{o}rmer-Verlet based scheme:
\begin{equation} \label{SVGLA}
\begin{cases}
& \Tilde{\boldsymbol{Q}}_{k+1} = \Tilde{\boldsymbol{Q}}_{k} 
+ h \boldsymbol{M}^{-1} e^{-\gamma \boldsymbol{M}^{-1} h/2}  \Tilde{\boldsymbol{P}}_{k}  
- \frac{h^2}{2} \boldsymbol{M}^{-1} \nabla U(  \Tilde{\boldsymbol{Q}}_{k} ) \\
& \qquad \qquad + h \sqrt{2 \beta^{-1} \gamma} 
\int_{t_k}^{t_k+h/2} \boldsymbol{M}^{-1} e^{- \gamma \boldsymbol{M}^{-1} (t_k+h/2-s) } d\boldsymbol{W}(s) \text{,} \\
& \Tilde{\boldsymbol{P}}_{k+1} =  e^{-\gamma \boldsymbol{M}^{-1} h}  \Tilde{\boldsymbol{P}}_{k}  
- \frac{h}{2} e^{- \gamma \boldsymbol{M}^{-1} h/2} \left( \nabla U( \Tilde{\boldsymbol{Q}}_k ) 
+ \nabla U( \Tilde{\boldsymbol{Q}}_{k+1} ) \right) \\
& \qquad \qquad + \sqrt{2 \beta^{-1} \gamma} 
\int_{t_k}^{t_k+h} e^{- \gamma \boldsymbol{M}^{-1} (t_k + h -s ) } d \boldsymbol{W}(s) \text{.}
\end{cases}
\end{equation}
It is straightforward to show \eqref{SVGLA} possesses a smooth probability transition function 
with respect to Lebesgue measure even though the noise is only applied to momenta in 
\eqref{OrnsteinUhlenbeck}.    For globally Lipschitz potential forces, it is 
possible to show that \eqref{SVGLA} is a first-order strongly accurate integrator for \eqref{SDE2}, 
and can be shown to be geometrically ergodic with respect to an invariant measure that is a first-order accurate 
approximation to the exact invariant measure of \eqref{SDE2} \cite{BoOw2009}.  However, if the potential force
is nonglobally Lipschitz, then this discretization is plagued with the same transient behavior 
as forward Euler-Maruyama applied to the reversible Langevin diffusion process.  As before, a 
Metropolis-Hastings method is proposed to correct the discretization error in the invariant 
measure and stochastically stabilize this discretization.

\paragraph{MAGLA.}

Let $\zeta_k \sim U(0,1)$ for $k=0,..., N-1$.  To Metropolize GLA \eqref{GLA}, 
we adopt the method of Metropolizing an inertial Langevin integrator introduced in 
\cite{ScLeStCaCa2006}.   In that paper, the authors present and test this Metropolization 
technique using a potpourri of integrators including the Ricci-Ciccotti algorithm as a 
proposal move \cite{RiCi2003}.   The method presented in Scemama et al.~involves the involution 
$\varphi: \mathbb{R}^{2n} \to \mathbb{R}^{2n}$:
\[
\varphi(\boldsymbol{q}, \boldsymbol{p}) = (\boldsymbol{q},-\boldsymbol{p})  \text{.}
\]   
This involution is introduced since the stochastic process defined by composing the solution of \eqref{SDE2} 
with $\varphi$ is reversible with respect to $\mu$.

The Metropolis-Adjusted Geometric Langevin Algorithm (MAGLA) will be based 
on the probability transition density of GLA which we will refer to as $q_h$.   An 
explicit formula for this transition density is derived in the body of the paper 
(See \eqref{GLAdensity}.). Given $h$ and $(\boldsymbol{Q}_k, \boldsymbol{P}_k)$, 
MAGLA computes a proposal move $(\boldsymbol{Q}_{k+1}^*, \boldsymbol{P}_{k+1}^*)$ 
according to a step of GLA composed with a momentum flip.   For example,
for the St\"{o}rmer-Verlet based GLA this proposal move is given explicitly by:
\begin{equation} \label{SVGLAproposal}
\begin{cases}
& \boldsymbol{Q}_{k+1}^* = \boldsymbol{Q}_{k} + 
h \boldsymbol{M}^{-1} e^{-\gamma \boldsymbol{M}^{-1} h/2}  \boldsymbol{P}_{k}  
- \frac{h^2}{2} \boldsymbol{M}^{-1} \nabla U(  \boldsymbol{Q}_{k} ) \\
& \qquad \qquad + h \sqrt{2 \beta^{-1} \gamma} 
\int_{t_k}^{t_k+h/2} \boldsymbol{M}^{-1} e^{- \gamma \boldsymbol{M}^{-1} (t_k+h/2-s) } d\boldsymbol{W}(s) \text{,} \\
& \boldsymbol{P}_{k+1}^* =  e^{-\gamma \boldsymbol{M}^{-1} h} \boldsymbol{P}_{k}
- \frac{h}{2} e^{- \gamma \boldsymbol{M}^{-1} h/2} \left( \nabla U( \boldsymbol{Q}_{k}) + \nabla U( \boldsymbol{Q}_{k+1}^* ) \right) \\
& \qquad \qquad + \sqrt{2 \beta^{-1} \gamma} 
\int_{t_k}^{t_k+h} e^{- \gamma \boldsymbol{M}^{-1} (t_k + h -s ) } d \boldsymbol{W}(s) \text{.}
\end{cases}
\end{equation}
The algorithm accepts this proposal move with probability:
\begin{equation} \label{MAGLAacceptreject}
\alpha_h((\boldsymbol{q}_0,\boldsymbol{p}_0), (\boldsymbol{q}_1,\boldsymbol{p}_1)) = 1 \wedge 
\frac{q_h ((\boldsymbol{q}_1,\boldsymbol{p}_1), (\boldsymbol{q}_0,-\boldsymbol{p}_0))  \pi(\boldsymbol{q}_1,\boldsymbol{p}_1)}
        {q_h ((\boldsymbol{q}_0,\boldsymbol{p}_0), (\boldsymbol{q}_1,-\boldsymbol{p}_1)) \pi(\boldsymbol{q}_0,\boldsymbol{p}_0)}
\end{equation}
In other words,  the MAGLA update is defined as:
\begin{align} \label{MAGLA}
& \boldsymbol{X}_{k+1} :=  ( \boldsymbol{Q}_{k+1}, \boldsymbol{P}_{k+1} ) = \nonumber \\ 
& \qquad \begin{cases}
(\boldsymbol{Q}_{k+1}^*, \boldsymbol{P}_{k+1}^*)  \qquad &  
\text{if}~~\zeta_k< \alpha_h((\boldsymbol{Q}_k,\boldsymbol{P}_k), (\boldsymbol{Q}_{k+1}^*,\boldsymbol{P}_{k+1}^*)) \\
\varphi(\boldsymbol{Q}_k, \boldsymbol{P}_k) \qquad &  \text{otherwise}
\end{cases} 
\end{align}
for $k=0,...,N-1$.  Observe, when a proposal move is rejected, the momentum is ``flipped''; and 
MAGLA involves a modified detailed balance condition \eqref{MAGLAacceptreject} as opposed
to the usual detailed balance condition used with MALA \eqref{MALAacceptreject}.  The reason
MAGLA uses a modified detailed balance is related to the nonreversibility of the solution to
\eqref{SDE2}.  In particular, the transition density of the solution to \eqref{SDE2} 
does not satisfy detailed balance, but does satisfy a modified detailed balance condition.   
It is straightforward to show that MAGLA preserves $\mu$, and hence, is not transient.  
In fact, it is often easy to classify MAGLA as an ergodic Markov chain even if the potential force
is nonglobally Lipschitz.   In the next paragraph, we quantify the strong convergence of MAGLA.

\paragraph{Main Result II: Strong Convergence of MAGLA.}

It is straightforward to show GLA provides a first-order globally accurate integrator for 
\eqref{SDE2} provided the potential force is globally Lipschitz \cite{BoOw2009}.  
As discussed MAGLA involves a momentum flip whenever a proposal move is rejected.  
This momentum flip was introduced since the stochastic process defined by 
composing the solution of \eqref{SDE2} with a momentum flip satisfies detailed balance with 
respect to $\pi$.  Despite this momentum flip, MAGLA provides a pathwise approximant to 
\eqref{SDE2} as summarized  by the following theorem.   As in Theorem~\ref{MALAaccuracy},
we will restrict the initial conditions of MAGLA to be distributed according to the equilibrium
distribution of \eqref{SDE2}.   As before this assumption implies bounds on relevant moments of 
MAGLA which follow from the fact that MAGLA preserves the measure $\mu$.
\begin{theorem}[St\"{o}rmer-Verlet-Based MAGLA Strong Accuracy from Equilibrium]
Assume \ref{sa} (C) and \ref{sa} (E) on the potential energy and \ref{sa2}  on \eqref{SDE2}.    
Then for every $T>0$,  there exists $h_c > 0$ and $C(T)>0$ such that for all
  $h<h_c$, for all $\boldsymbol{x} \in \mathbb{R}^{2n}$, and for all $t \in [0, T]$,
\[
\left( \E_{\mu}  \E^{\boldsymbol{x}} 
\left\{ \left| \boldsymbol{X}_{\lfloor t/ h \rfloor} - \boldsymbol{Y}(t) \right|^2 \right\} 
\right)^{1/2} \le C(T) h \text{.}
\]
\label{SVMAGLAaccuracy}
\end{theorem}
Assumptions  \ref{sa} (C) and \ref{sa} (E) on the potential energy hold even if potential
force is nonglobally Lipschitz.  The assumption \ref{sa2}  on \eqref{SDE2}
ensures that the continuous process itself is globally Lipschitz.  We stress that this regularity 
of the continuous process is different from assuming $\nabla U$ is globally Lipschitz, 
and is made for convenience.

Unlike the rejections in \eqref{MALA}, a rejection in \eqref{MAGLA} leads to a momentum flip,
and hence, a loss of local accuracy.  Even though, MAGLA still approximates pathwise the 
solution to \eqref{SDE2} because the probability of these rejections as a function of 
time-step size is small.   These issues including a proof of Theorem~\ref{SVMAGLAaccuracy} and
the lemmas on which it is based can be found in \S \ref{sec:InertialLangevin}.


\newpage


\section{Numerical illustration} \label{sec:NumericalIllustration}

\begin{figure}[ht!]
\begin{center}
\includegraphics[scale=0.75,angle=0]{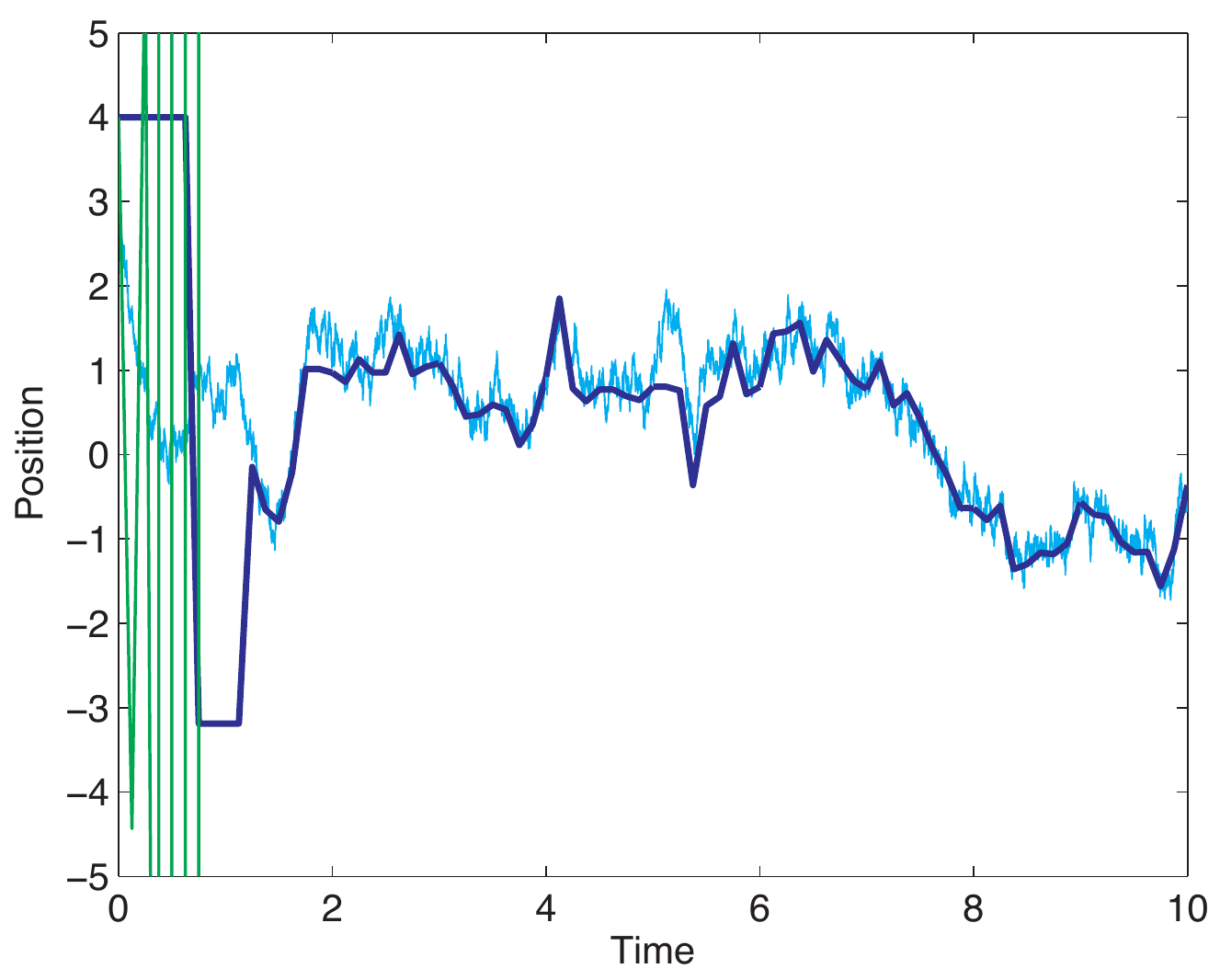}
\caption{ \small {\bf Exploding, Stagnating, and Contracting Orbits.} 
	For $h=0.3125$, $\beta=1.0$ and $x=4$, this figure illustrates a realization of the 
	solution to \eqref{cosde} (cyan), the Euler-Maruyama approximation  \eqref{coeuler} (green), 
	and MALA (blue) initiated at the edge of $B_h$ and driven by the 
	same realization of the Wiener process $W$.   Observe that Euler-Maruyama is explosive 
	while the Metropolized orbit is not.   Also observe that the position of the exact solution 
	contracts rapidly while the Metropolized orbit stagnates before contracting due to 
	proposal moves being rejected.   
} 
\label{fig:explodingstagnatingcontracting} \end{center}
\end{figure}

\begin{figure}[ht!]
\begin{center}
\includegraphics[scale=0.75,angle=0]{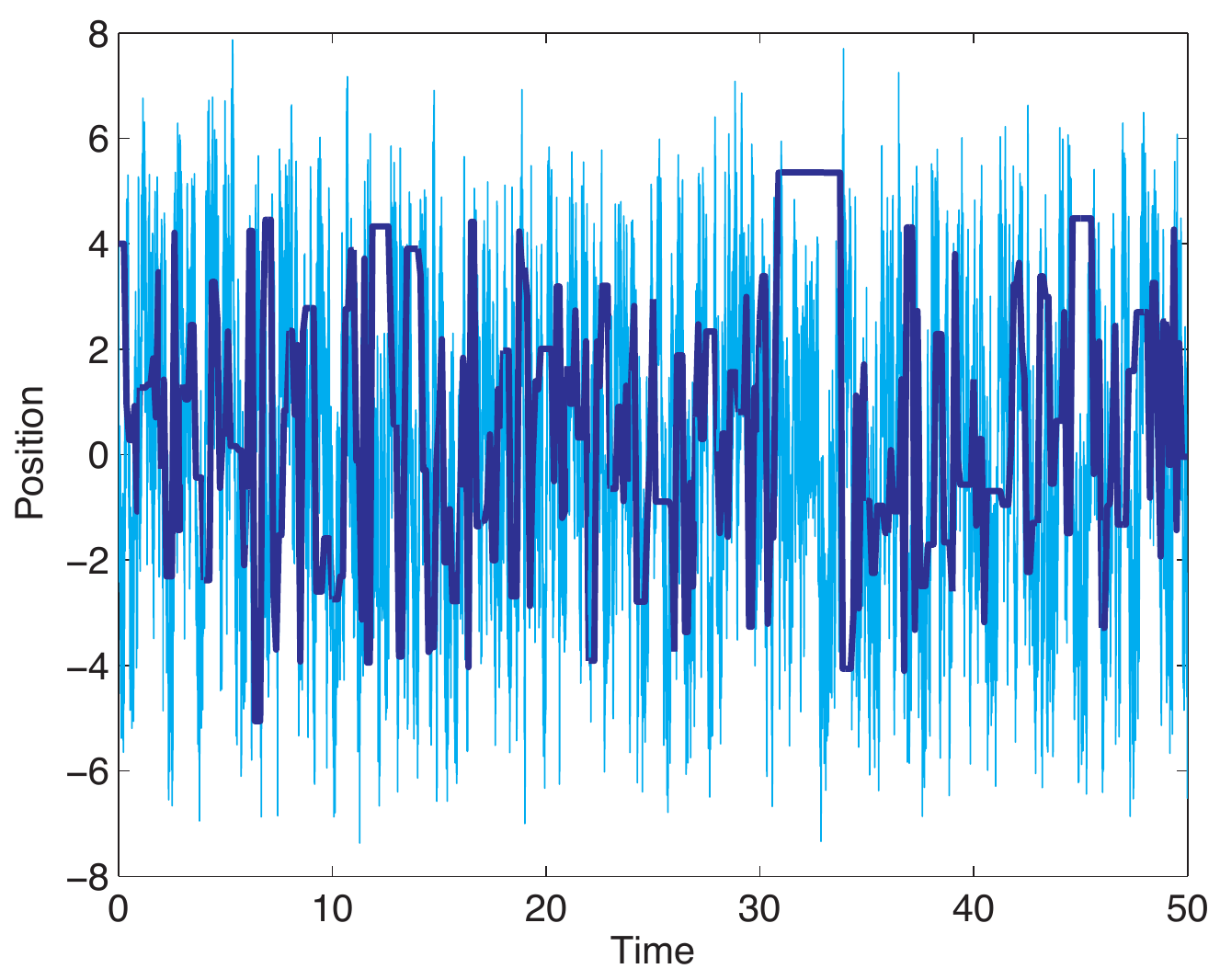}
\caption{ \small {\bf Long-Time Behavior of MALA.} 
	For $h=0.3125$, $\beta=0.01$ and $x=4.0$, a long-time realization of the solution 
	to \eqref{cosde} (cyan) and MALA (blue).  Observe that the 
	solution frequently visits higher energy values that are not reached by the 
	Metropolized integrator.   But, ergodicity requires that the Metropolized integrator 
	visit these higher energy levels and for about the same ratio of time spent by the 
	solution, since both chains sample from $\pi$.  Hence, the Metropolized integrator 
	intermittently stagnates at higher energy values as shown in the figure.  
} 
\label{fig:longtimeorbit} \end{center}
\end{figure}

\begin{figure}[ht!]
\begin{center}
\includegraphics[scale=0.75,angle=0]{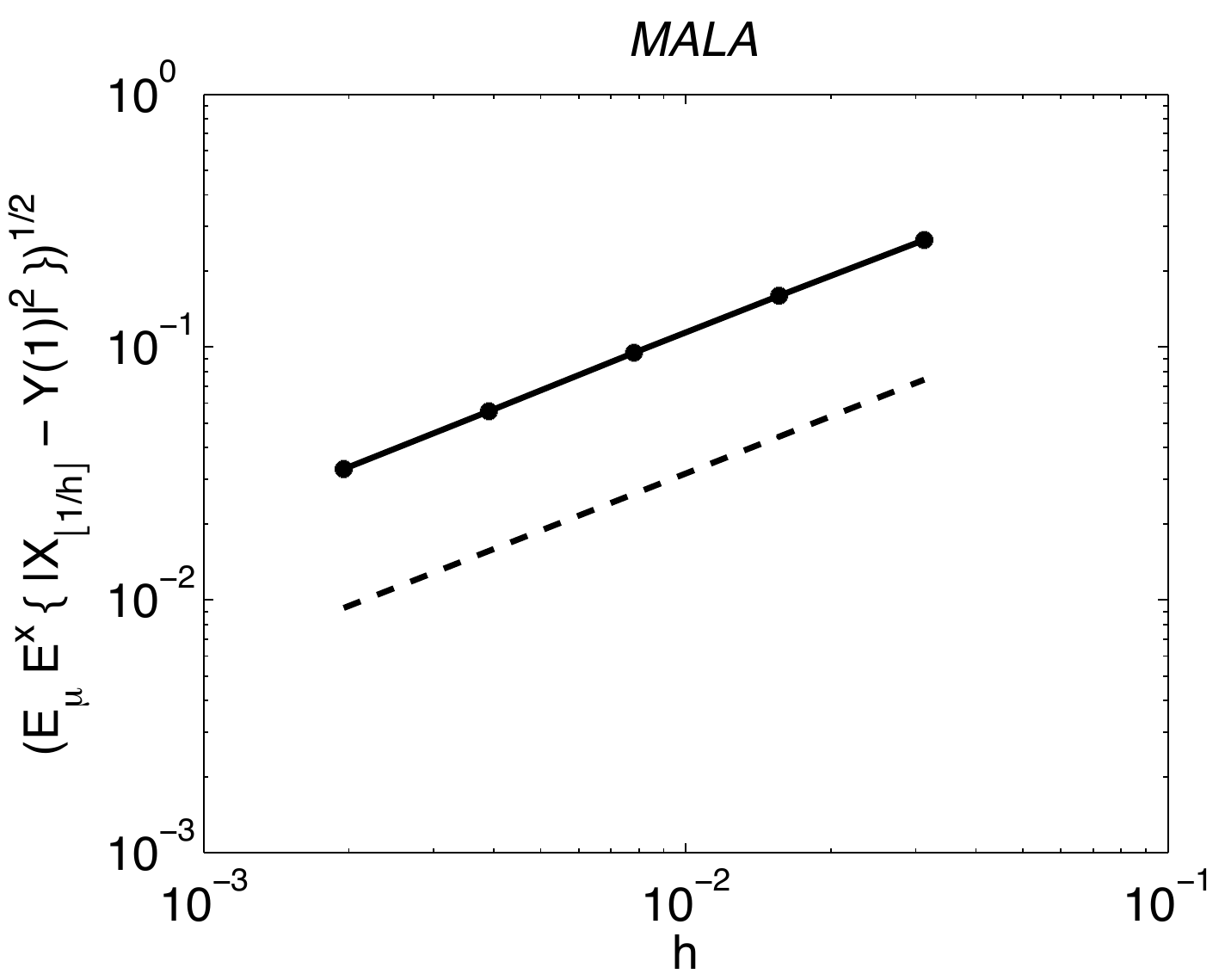}
\caption{ \small {\bf Strong Accuracy of MALA.}  
	This loglog plot illustrates that despite random rejections of integration steps and the
	nonglobally Lipschitz nature of the drift, MALA remains mean-squared, and hence strongly, 
	convergent to solutions of the SDE.   The dashed line represents a reference slope of $h^{3/4}$.  
	The solid line represents an empirical  estimate of the mean-squared order of accuracy 
	of MALA for an inverse temperature $\beta=1$,  time-span of $T=1$, and an ensemble
	of initial conditions drawn from the equilibrium distribution of the SDE using inverse transform
	sampling.  The estimated rate of convergence agrees with that reported in 
	Theorem~\ref{MALAaccuracy}.  A total of $10^6$ realizations were used to 
	obtain this estimate. 
  }
\label{fig:MALAPathwiseAccuracy} \end{center}
\end{figure}

\begin{figure}[ht!]
\begin{center}
\includegraphics[scale=0.75,angle=0]{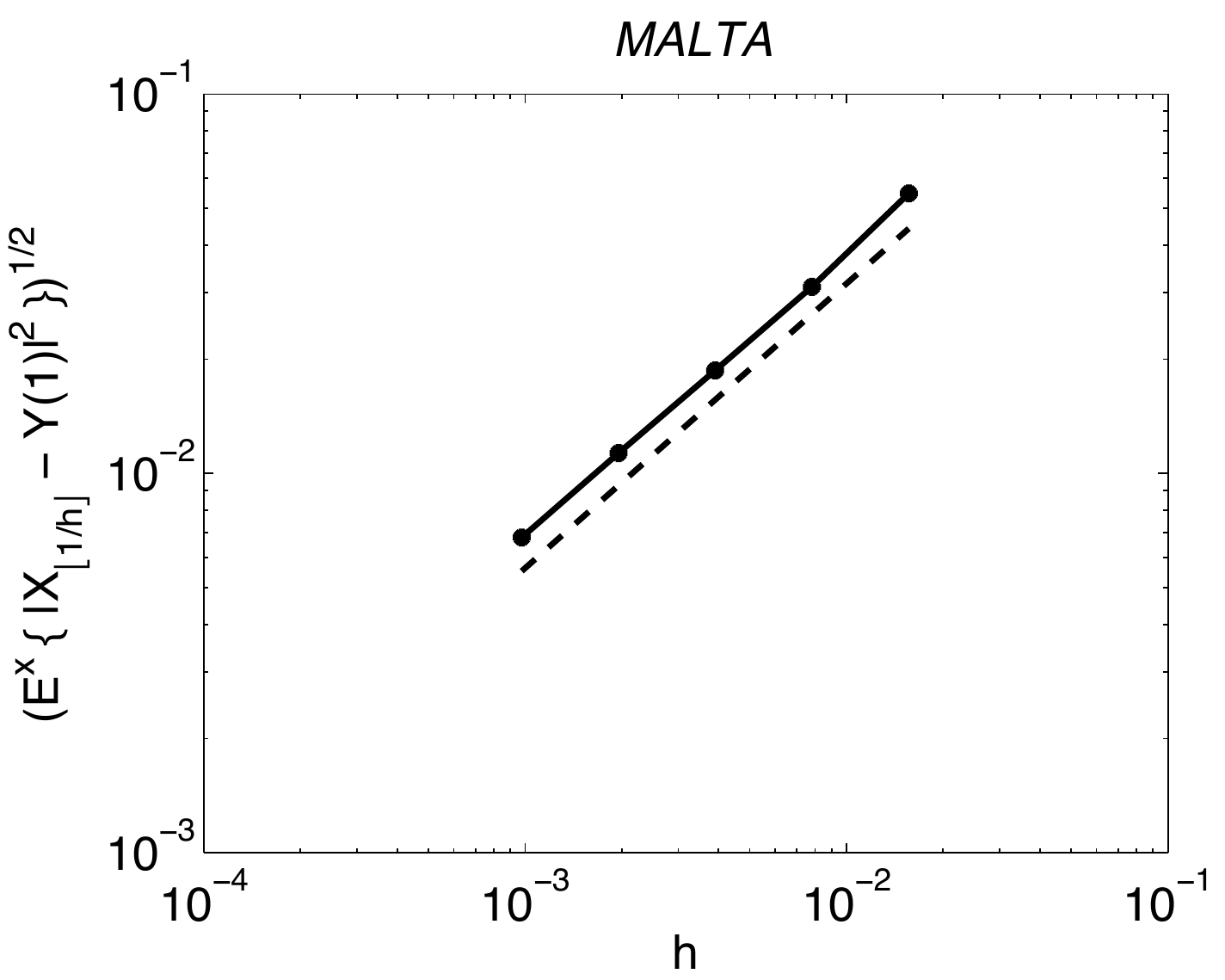}
\caption{ \small {\bf Strong Accuracy of  MALTA.}  
	This loglog plot illustrates that despite random rejections of integration steps and the
	{\em ad hoc} truncation of the drift, MALTA remains mean-squared, and hence strongly, 
	convergent to solutions of the SDE.   The dashed line represents a reference slope of $h^{3/4}$.  
	The solid line represents an empirical  estimate of the mean-squared order of accuracy 
	of MALTA for an inverse temperature $\beta=1$,  time-span of $T=1$, and initial 
	condition $x=0.1$.  The estimated rate of convergence agrees with that reported in
	Theorem~\ref{MALTAaccuracy}.  A total of $10^6$ realizations were used to 
	obtain this estimate. 
  }
\label{fig:MALTAPathwiseAccuracy} \end{center}
\end{figure}

\subsection{Overdamped Langevin}

Here we illustrate the results above on the simple example
of a particle diffusing in a quartic potential $U(x)=x^4/4$ with inverse
temperature $\beta>0$.  The overdamped dynamics of this system is given by:
\begin{equation}
d Y =  - Y^3  dt + \sqrt{2 \beta^{-1}} dW,~~Y(0) = x.
\label{cosde}
\end{equation}
This SDE admits the following unique invariant measure 
\[
\mu(dx) = Z^{-1} \exp(-\beta x^4/4) dx \text{.}
\]
Let $N$ and $h$ be given.  Set $T= N h$ and $t_k = h k $ for
$k=0,...,N$.  In terms of which consider the following forward Euler
approximation to \eqref{cosde}:
\begin{equation}
  \Tilde X_{k+1} = \Tilde X_k - h \Tilde X_k^3  
  + \sqrt{2 \beta^{-1}} (W(t_{k+1})-W(t_k)),
  \qquad \Tilde X_0 = x.
\label{coeuler}
\end{equation}
The drift in \eqref{coeuler} is destabilizing in the region:
\[
B_h = \{ x : |1- h x^2|>1 \} \text{.}
\]   
This property of Euler-Maruyama's discrete drift is the essential reason why 
Euler-Maruyama defines a transient Markov chain.  Indeed, using this property 
\eqref{emunboundedmoments} is proved in Lemma 6.3 of \cite{MaStHi2002}.  
Despite this shortcoming Euler-Maruyama can be used as candidate dynamics 
in a Metropolis-Hastings method designed to sample from $\mu$.

Figure~\ref{fig:explodingstagnatingcontracting} illustrates the
difference between a given realization of the exact solution to
\eqref{cosde}, the Euler-Maruyama integrator \eqref{coeuler}, and the
Metropolized integrator when initiated at the edge of $B_h$ and driven
by the same realization of the Wiener process $W$.   We empasize that in this figure
the time-step is held fixed and chosen to be large.  The figure shows
the Euler-Maruyama orbit explodes, the solution orbit rapidly
contracts to the origin, and the Metropolis-Hastings method initially
stagnating, but eventually tracking the solution apparently better over time.   
We remark that the ability of the Metropolized orbit to track the solution 
better over time is not a generic property of Metropolized integrators.  In this case it
is a consequence of  \eqref{cosde} satisfying a  {\em ``one force, one solution"}  principle 
(See \cite{EKhMaSi2000}.).  That is, for every realization of the Wiener process, the solution 
possesses a random attractor that the Metropolized orbit is drawn to.

Figure~\ref{fig:longtimeorbit} shows a longer-time realization of the Metropolized 
integrator and the solution.   The solution at this temperature frequently visits higher 
energy values.  The proposal moves to higher energy values are less likely to be
accepted by the Metropolized integrator.  At the same time, ergodicity implies the Metropolized integrator 
must visit these higher energy values, and as often as the exact solution along this long time-interval, 
since both chains preserve $\pi$.  Consequently, the Metropolized integrator intermittently stagnates 
as shown in the figure, and of course, loses pathwise accuracy.

This observation does not contradict Theorems~\ref{MALTAaccuracy} and~\ref{MALAaccuracy}.
Keep in mind that the time-step size for this particular orbit is held fixed and chosen to be large.  
Such stagnations at high energy become less likely to occur along orbits on finite-time intervals 
as the time-step becomes smaller.  Figures~\ref{fig:MALTAPathwiseAccuracy} and 
\ref{fig:MALAPathwiseAccuracy} confirms this pathwise convergence.  In particular, the figures 
show an $\mathcal{O}(h^{3/4})$ rate of strong convergence of the Metropolized integrator for this 
example using MALTA and an out of equilibrium initial condition, and MALA with an initial 
condition restricted to the equilibrium distribution of \eqref{cosde}.


\begin{figure}[ht!]
\begin{center}
\includegraphics[scale=0.75,angle=0]{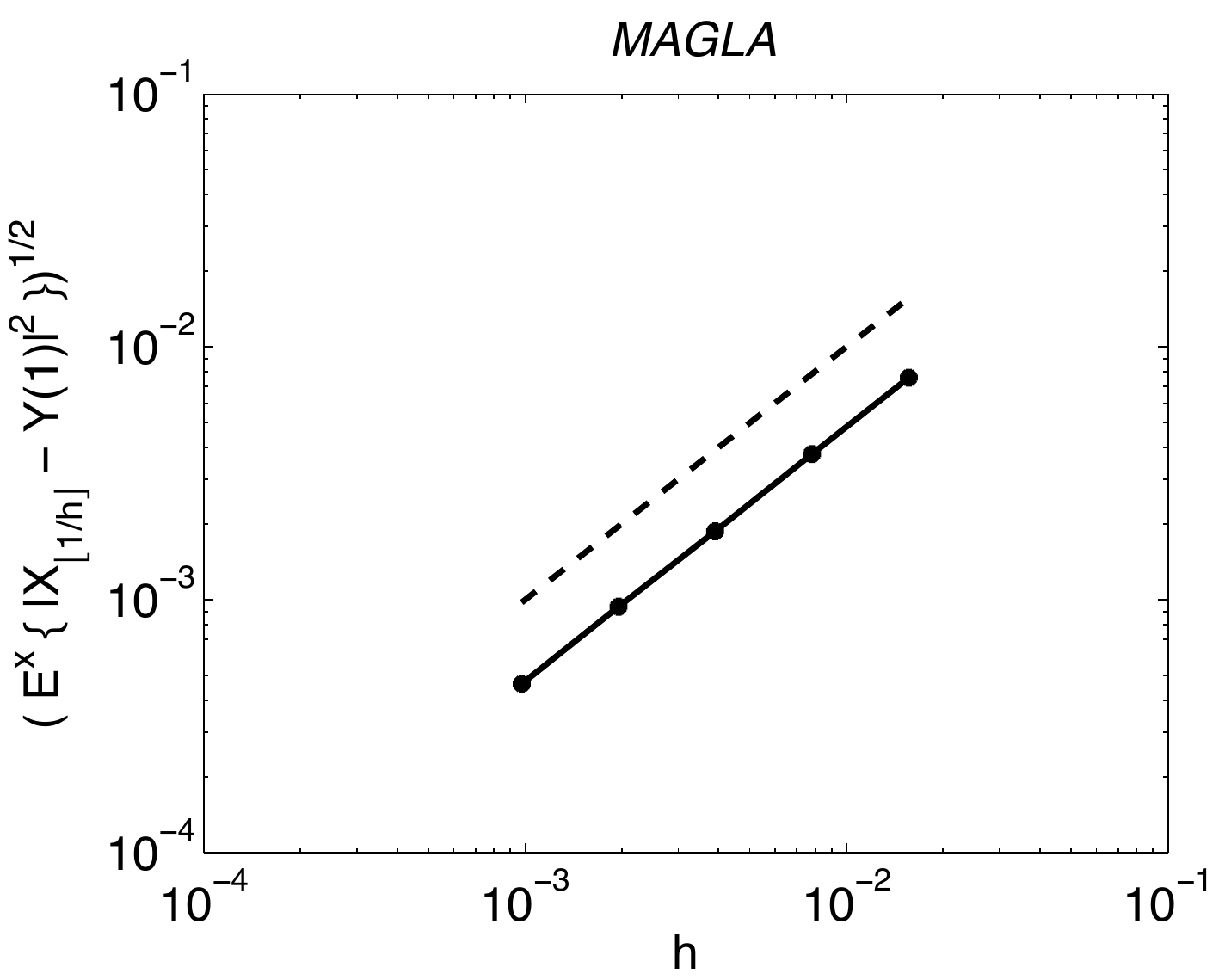}
\caption{ \small {\bf Strong Accuracy of MAGLA.}   
	This plot illustrates  that despite momentum flips which occur at every rejection of a proposal move, 
	and the nonglobally Lipschitz nature of the drift, MAGLA remains mean-squared, and hence strongly, 
	convergent to solutions of the SDE.  In this loglog plot, the dashed line represents a reference slope of $h$.
	The solid line shows a sample average of the mean-squared order of accuracy of MAGLA,  for $T=1$, 
	$\beta=\gamma=1$ and $(q_0,p_0)=(0.1,0.0)$.    The estimated rate of convergence is consistent 
	with that reported in Theorem~\ref{SVMAGLAaccuracy}.  A total of $10^5$ realizations were
	used to obtain this plot. 
}
\label{fig:MAGLAPathwiseAccuracy} \end{center}
\end{figure}

\subsection{Inertial Langevin}

Consider a simple particle with Hamiltonian given by:
\[
H(q, p) = p^2/2 + U(q)
\]
where $U(q)=q^4/4$.  The inertial Langevin dynamics of this system is given by:
\begin{equation} \label{icosde}
\begin{cases}
 d Q &=  P dt \\
 d P &= - Q^3  dt - \gamma P dt + \sqrt{2 \gamma \beta^{-1}} dW,
\end{cases}
\end{equation}
with initial condition $(Q(0), P(0)) = (q_0, p_0)$.
This SDE admits the following unique invariant measure 
\[
\mu(dq dp) = Z^{-1} \exp(-\beta H(q,p)) dq dp \text{.}
\]
Let $N$ and $h$ be given.  Set $T= N h$ and $t_k = h k $ for
$k=0,...,N$.  In terms of which consider the following St\"{o}rmer-Verlet
based GLA discretization of \eqref{icosde}:
\begin{equation} \label{cogla}
\begin{cases}
&  \Tilde Q_{k+1} = \Tilde Q_k + h e^{- \gamma h/2} \Tilde P_k - \frac{h^2}{2} \Tilde X_k^3  \\
 & \qquad \qquad  + \sqrt{2 \beta^{-1} \gamma} \int_{t_k}^{t_k+h/2} e^{-\gamma (t_k+h/2-s) } d W(s),  \\
& \Tilde P_{k+1} = e^{-\gamma h} \Tilde P_k 
- \frac{h}{2} e^{-\gamma h/2} \left( \nabla U(\Tilde Q_k) + \nabla U(\Tilde Q_{k+1} ) \right) \\
& \qquad \qquad + \sqrt{2 \beta^{-1} \gamma} \int_{t_k}^{t_k+h} e^{-\gamma (t_k+h-s) } d W(s) \text{.}
\end{cases}
\end{equation}
with initial condition $(\Tilde Q_0, \Tilde P_0) = (q_0, p_0)$.  
The inertial Langevin integrator \eqref{cogla} is plagued with the 
same transient behavior as forward Euler-Maruyama applied to the 
reversible cubic oscillator Langevin dynamics.  This occurs
at points in phase space where the underlying explicit 
St\"{o}rmer-Verlet integrator in \eqref{cogla} becomes 
linearly unstable.

A Metropolis-Hastings method can stochastically
stabilize \eqref{cogla}.  At the same time, as stated in 
Theorem~\ref{SVMAGLAaccuracy}, the method is also
pathwise accurate with respect to the solution of \eqref{icosde}
when initiated from equilibrium.  This accuracy is achieved despite the 
loss of local accuracy in momentum that occurs at every rejection.

Figure~\ref{fig:MAGLAPathwiseAccuracy} is consistent with the strong $\mathcal{O}(h)$ 
rate of convergence of MAGLA for this example.  Assumption~\ref{sa2} is hard to
check for this example, but we emphasize the assumption is not equivalent 
to the potential force being globally Lipschitz.


\newpage


\section{Overdamped Langevin}  \label{sec:OverdampedLangevin}


\subsection{Preliminaries}  \label{preliminaries}

\paragraph{Structural Assumptions.}

For a function $G \in C^{\infty}(\mathbb{R}^n, \mathbb{R})$ and an integer $r>1$, 
let $\nabla G$ and $D^r G$ be the gradient and the $rth$-derivative of $G$, respectively.       
Let $| \cdot |$ denote the Euclidean vector norm and $\| \cdot \|$ the Frobenius 
norm.  Let $L$ denote the generator of \eqref{SDE1} defined for any 
$G \in C^{2}( \mathbb{R}^n, \mathbb{R})$ as
 \begin{align}
 L G( \boldsymbol{x} ) &= - \nabla U(\boldsymbol{x}) \cdot \nabla G(\boldsymbol{x}) 
 + \beta^{-1} \operatorname{trace}[ D^2 G(\boldsymbol{x}) ] \text{.} 
 \label{generator}
 \end{align}

\begin{Assumption}
 The following structural assumptions will be invoked in this paper regarding \eqref{SDE1}:
\begin{description}
\item[A)]  
There exists a real constant $K>0$ such that
\[
U( \boldsymbol{x} ) \ge K   | \boldsymbol{x} | \text{,}
~~\forall~ \boldsymbol{x} \in \mathbb{R}^n\text{.}
\]
\item[B)] 
For every integer $\ell \ge 1$ there exist real constants $\delta_{\ell}>0$ and $M_{\ell} >0$ such that
the operator \eqref{generator} applied to the $\ell$th-power of the potential energy satisfies:
\[
L \{ U(\boldsymbol{x} )^{\ell} \} \le - \delta_{\ell} U(\boldsymbol{x} )^{\ell}  + M_{\ell} \text{,}
~~\forall~ \boldsymbol{x} \in \mathbb{R}^n\text{.}
\]
\item[C)]
There exists a real constant $K>0$ such that
\[
\left| \nabla U( \boldsymbol{x} ) - \nabla U( \boldsymbol{y} ) \right| \le 
K ( U( \boldsymbol{x} ) + U( \boldsymbol{y} ) ) | \boldsymbol{x} - \boldsymbol{y} |,
~~\forall~\boldsymbol{x}, \boldsymbol{y} \in \mathbb{R}^n \text{.}
\]
\item[D)]
There exists a real constant $K>0$ such that
\[
\left\langle - \nabla U( \boldsymbol{x} ) + \nabla U( \boldsymbol{y} ) , \boldsymbol{x} - \boldsymbol{y} \right\rangle  
\le  K | \boldsymbol{x} - \boldsymbol{y} |^2,
~~\forall~\boldsymbol{x}, \boldsymbol{y} \in \mathbb{R}^n \text{.}
\]
\item[E)]
There exists a real constant $K>0$ such that
\[
\left\| D^3 U( \boldsymbol{x} ) \right\| \vee \left\| D^2 U( \boldsymbol{x} ) \right\| \vee  \left| \nabla U( \boldsymbol{x} ) \right| 
 \le K (1 + U( \boldsymbol{x} ) ),
 ~~\forall~\boldsymbol{x} \in \mathbb{R}^n \text{.}
\]
\end{description}
\label{sa}
\end{Assumption}

We stress these structural assumptions hold even if the drift in \eqref{SDE1} is nonglobally Lipschitz.
These structural assumptions imply that the Markov process defined by the solution to
\eqref{SDE1} possesses a unique invariant probability measure explicitly given by:
\begin{equation}
\mu(d \boldsymbol{x}) :=\pi(\boldsymbol{x}) d \boldsymbol{x} 
\label{pi}
\end{equation}
where $\pi(\boldsymbol{x}) = Z^{-1} \exp(- \beta U( \boldsymbol{x} )) $
with $Z = \int_{\mathbb{R}^n} \exp(- \beta U(\boldsymbol{x} )) d\boldsymbol{x} $.   
These structural assumptions also imply the following estimates on higher moments
of solutions to \eqref{SDE1} which we state as a lemma.
\begin{lemma}[Estimates on Higher Moments of Solution]
Assume \ref{sa}.  For every integer $\ell \ge 1$ we have the following estimates 
on higher moments of the solution to \eqref{SDE1}.
\begin{description}
\item[A)]   Given $M_{\ell}>0$ and $\delta_{\ell}>0$ from assumption~\ref{sa}~(B),
\[
\E^{\boldsymbol{x}} \left\{ U( \boldsymbol{Y}(t)  )^{\ell} \right\} \le 
U(\boldsymbol{x})^{\ell} + \frac{M_{\ell}}{\delta_{\ell}},
~~\forall~\boldsymbol{x} \in \mathbb{R}^n,~~\forall ~t>0 \text{.}
\]
\item[B)]  There exists $K_{\ell} > 0$ such that
\[
\E^{\boldsymbol{x}} \left\{ \left| \boldsymbol{Y}(t) \right|^{\ell} \right\}  \le 
K_{\ell} ( 1 +  U( \boldsymbol{x} )^{\ell} ),
~~\forall~\boldsymbol{x} \in \mathbb{R}^n,~~\forall ~t>0 \text{.}
\]
\item[C)] There exists $K_{\ell}>0$ such that
\[
\E^{\boldsymbol{x}} \left\{ \left| \boldsymbol{Y}(t) - \boldsymbol{x} \right|^{2 \ell} \right\} \le 
K_{\ell}  (1+t^{\ell} U( \boldsymbol{x} )^{ 2 \ell}) t^{ \ell},
~~\forall~\boldsymbol{x} \in \mathbb{R}^n,~~\forall ~t>0  \text{.}
\]
\end{description}
\label{sdemomentbound}
\end{lemma}
The proof of this lemma is presented in \S\ref{sec:prooflemma}.

The structural assumptions on \eqref{SDE1} also imply a Lipschitz condition 
on its solution which we state as another lemma.
\begin{lemma}[Regularity of Solutions]
Assume \ref{sa}. For $s \le t$,  let $\boldsymbol{Y}_{t,s}( \boldsymbol{x})$ denote 
the evolution operator of the solution to \eqref{SDE1}: 
with $\boldsymbol{Y}_{s,s}(\boldsymbol{x}) = \boldsymbol{x}$ and for  $r \le s \le t$ 
recall the Chapman-Kolmogorov identity 
$\boldsymbol{Y}_{t,s} \circ \boldsymbol{Y}_{s,r}(\boldsymbol{x}) = \boldsymbol{Y}_{t,r}(\boldsymbol{x})$.   
Set
\[
\boldsymbol{\Delta} =  
\boldsymbol{Y}_{s+h,s}(\boldsymbol{x}) - \boldsymbol{Y}_{s+h,s}(\boldsymbol{y}) 
- ( \boldsymbol{x} -\boldsymbol{y} ) \text{,}
\]
For all $K>0$ there exists $h_c>0$, such that for all positive $h<h_c$, for all 
$\boldsymbol{x}, \boldsymbol{y} \in \mathbb{R}^n$, and for all $s\ge0$,
\begin{description}
\item[A)]  
\[ 
\E \{ | \boldsymbol{Y}_{s+h,s}(\boldsymbol{x}) - \boldsymbol{Y}_{s+h,s}(\boldsymbol{y})  |^2 \} 
\le | \boldsymbol{x} - \boldsymbol{y} |^2  ( 1+  K h) \text{;}
\]
\item[B)] 
\[ 
\E \{  | \boldsymbol{\Delta}  |^2 \} 
\le   K h^2  ( 1 + U(\boldsymbol{x})^2 + U(\boldsymbol{y})^2 )  | \boldsymbol{x} - \boldsymbol{y} | \text{.}
\]
\end{description}
\label{regularityofsolutions}
\end{lemma}
The proof of this lemma is also given below in section~\ref{sec:prooflemma}.

\paragraph{Explicit Integrator for the SDE.}

Recall that the transition kernel for ULA (cf.~\eqref{ULA}) has a density
$q_h(\boldsymbol{x}, \boldsymbol{y})$ with respect to the Lebesgue
measure given by
\begin{equation}
q_h(\boldsymbol{x}, \boldsymbol{y}) = (4 \pi \beta^{-1} h)^{-n/2} 
\exp\left(  -  \frac{\left| \boldsymbol{y} - \boldsymbol{x} 
+ h  \nabla U(\boldsymbol{x})  \right|^2}{ 4 \beta^{-1} h} \right) \text{.}
\label{ULAqh}
\end{equation}
It is derived by a simple change of variables relating the probability
density of the Euler-Maruyama update to the probability density of a
multivariate Gaussian with zero mean and variance $2 \beta^{-1} h$.
ULA is irreducible with respect to the Lebesgue
measure because its transition density is smooth and strictly
positive.  However, as mentioned in the introduction, in Langevin
diffusions with nonglobally Lipschitz drift, ULA
is transient, i.e., stochastically unstable.  Nevertheless under the
assumptions on the drift, it is straightforward to obtain the
following single-step accuracy for the forward Euler-Maruyama scheme.
\begin{lemma}[Local Strong Accuracy of Euler-Maruyama]
Assume \ref{sa}.  For all $K>0$ there exists a $h_c$ such that for all 
positive $h<h_c$ and for all $\boldsymbol{x} \in \mathbb{R}^n$
\begin{description}
\item[A)] the local mean-squared error of \eqref{ULA} satisfies
\[
\E^{\boldsymbol{x}} \{ | \Tilde{\boldsymbol{X}}_1 -
    \boldsymbol{Y}(h) |^2 \} \le K (1 + U(
\boldsymbol{x})^4 ) h^{3} \text{;}
\]
\item[B)] the local mean deviation of \eqref{ULA} satisfies
\[
| \E^{\boldsymbol{x}} \{  \Tilde{\boldsymbol{X}}_1 -  \boldsymbol{Y}(h)   \} | 
\le K  (1 + U( \boldsymbol{x} )^2 ) h^2 \text{.}
\]
\end{description}
\label{ULAlocalaccuracy}
\end{lemma}
The above lemmas are proven in section~\ref{sec:prooflemma}.


\subsection{MALA and its Properties}

\paragraph{Ergodicity.}

MALA randomly rejects integration steps produced by Euler-Maruyama 
with a probability designed to ensure the composite chain preserves 
the equilibrium measure of \eqref{SDE1}.  Recall that this probability is given by:
\begin{equation}
\alpha_h(\boldsymbol{x}, \boldsymbol{y}) = 
1 \wedge \frac{q_h(\boldsymbol{y},\boldsymbol{x}) \pi(\boldsymbol{y}) }
		       {q_h(\boldsymbol{x},\boldsymbol{y}) \pi(\boldsymbol{x}) }  \text{.}
\label{acceptanceprobability}
\end{equation}
The off-diagonal transition probability kernel of the composite chain has density:
\[
p_h(\boldsymbol{x},\boldsymbol{y}) = 
q_h(\boldsymbol{x},\boldsymbol{y})  \alpha_h(\boldsymbol{x},\boldsymbol{y}) \text{.}
\]
The probability of remaining at the same point, or {\em stagnation probability}, 
is given by
\[
r_h(\boldsymbol{x}) = 
1 - \int_{\mathbb{R}^n} p_h(\boldsymbol{x},\boldsymbol{z}) d\boldsymbol{z}  \text{.}
\]
In sum, the Metropolis-Hastings transition kernel is
\begin{equation}
P_h(\boldsymbol{x},d \boldsymbol{y}) = 
p_h(\boldsymbol{x},\boldsymbol{y}) d \boldsymbol{y} + 
r_h(\boldsymbol{x}) \delta_{\boldsymbol{x}}(d\boldsymbol{y}) \text{.}
\label{mhkernel}
\end{equation}
where $\delta( \cdot )$ denotes the Dirac-delta measure in $\mathbb{R}^n$. 
It is straightforward to show that this $P_h$ is a well-defined probability 
transition kernel.  Let $\mathcal{B}(\mathbb{R}^n)$ denote the smallest $\sigma$-algebra 
containing all the open subsets of $\mathbb{R}^n$. For all $\boldsymbol{x} \in \mathbb{R}^n$ 
and $A \in \mathcal{B}(\mathbb{R}^n)$, the $k$-step transition probability kernel 
is iteratively defined as
\[
P_h^{k}( \boldsymbol{x}, A) = 
\int_{\mathbb{R}^n} P_h( \boldsymbol{x}, d \boldsymbol{y} ) P_h^{k-1}(\boldsymbol{y},A),~~~(k\ge1),
\]
with $P^0(\boldsymbol{x}, A)$ equal to the characteristic function of
the subset $A$ of $\mathbb{R}^n$, i.e., $ P^0(\boldsymbol{x}, A)=
\chi_A(\boldsymbol{x})$.

\begin{theorem}[MALA Ergodicity]
Under assumption~\ref{sa} on the potential energy, the $k$-step transition probability of 
MALA converges to $\mu$ in the following sense
\[
\sup_{A \in \mathcal{B}(\mathbb{R}^n)} | P_h^k(\boldsymbol{x}, A) - \mu(A) | \to 0,
~~\text{as}~k \to \infty,~~\forall~\boldsymbol{x} \in \mathbb{R}^n \text{.}
\]
\end{theorem}

\begin{proof}
  The following classification of  MALA
  is quite standard.  Since $\pi$ and $q_h$ are strictly positive and
  smooth everywhere, MALA is
  irreducible with respect to Lebesgue measure and aperiodic; see,
  e.g., Lemma 1.2 of~\cite{MenTw1996} and references therein.  By the
  design of the acceptance probability, MALA also preserves the 
  probability measure $\mu$.  To confirm this
  statement observe from \eqref{mhkernel} that for any $A \in
 \mathcal{B}(\mathbb{R}^n)$
\begin{align*}
\int_{\mathbb{R}^n} \mu(d\boldsymbol{x}) P_h(\boldsymbol{x},A)  =&  
\int_{\mathbb{R}^n} \left( \int_{A} \pi(\boldsymbol{x}) p_h(\boldsymbol{x},\boldsymbol{y}) d \boldsymbol{y} 
+ \pi(\boldsymbol{x}) r_h(\boldsymbol{x}) \chi _{A} (\boldsymbol{x}) \right) d\boldsymbol{x} \\
=& \int_{\mathbb{R}^n} \int_{A} \left( \pi(\boldsymbol{x})  p_h(\boldsymbol{x},\boldsymbol{y}) 
- \pi(\boldsymbol{y})  p_h(\boldsymbol{y},\boldsymbol{x})  \right)  d\boldsymbol{y}  d\boldsymbol{x} 
+ \mu(A)  
\end{align*}
However, since $p_h(\boldsymbol{x},\boldsymbol{y})$ satisfies detailed balance the first 
term in the above vanishes, and hence,
\[
\int_{\mathbb{R}^n} \mu(d\boldsymbol{x}) P_h(\boldsymbol{x},A)  =  \mu(A)
\]
According to Corollary 2 of \cite{Ti1994},  a Metropolis-Hastings algorithm that is 
irreducible with respect to the same measure it is designed to preserve 
is positive Harris recurrent.  Consequently, MALA 
is irreducible, aperiodic, and positive Harris recurrent.
According to Proposition 6.3 of~\cite{Nu1984}, these properties are 
equivalent to ergodicity of the chain.
\end{proof}

\begin{Remark}
This theorem shows that MALA is ergodic 
with respect to the equilibrium measure of \eqref{SDE1}.    
However, the effect of rejections on the
pathwise approximation of the solution remains to be quantified.  
\end{Remark}

\paragraph{Equilibrium Strong Accuracy.}

In order to prove equilibrium strong accuracy of MALA 
the following lemmas will be helpful.

\begin{lemma}[MALA Stagnation Probability]
Assume~\ref{sa}.   For all integers $\ell \ge 1$, there exists an $h_c>0$ and 
a constant $K_{\ell} >0$, such that for all positive $h<h_c$ and for all $\boldsymbol{x} \in \mathbb{R}^n$,
\[
\E^{\boldsymbol{x}} \left\{   ( \alpha_h(\boldsymbol{x},\boldsymbol{X}^*_1) - 1 )^{2 \ell}  \right\}  
\le K_{\ell} ( 1 + U(\boldsymbol{x})^{4 \ell}  ) h^{3 \ell} \text{,}
\]
where
\[
  \boldsymbol{X}^*_{1} = \boldsymbol{x} - h  \nabla U( \boldsymbol{x} ) + \sqrt{2 \beta^{-1} h} \boldsymbol{\eta} \text{.}
 \]
  and $\boldsymbol{\eta} \sim \mathcal{N}(0,1)^n$.
\label{MALAstagnationprobability}
\end{lemma}

\begin{proof}
Let $R_h(\boldsymbol{x}) = \{ \boldsymbol{y} \in \mathbb{R}^n : \alpha_h(\boldsymbol{x}, \boldsymbol{y}) < 1 \}$.  
Then,
\begin{align}
  \E^{\boldsymbol{x}} \left\{ (
    \alpha_h(\boldsymbol{x},\boldsymbol{X}^*_1) - 1 )^{2 \ell} \right\} =
  \int_{R_h(\boldsymbol{x})} \left( \frac{q_h(\boldsymbol{y},
      \boldsymbol{x}) \pi(\boldsymbol{y})}{q_h(\boldsymbol{x},
      \boldsymbol{y}) \pi(\boldsymbol{x})} - 1 \right)^{2 \ell}
  q_h(\boldsymbol{x}, \boldsymbol{y}) d\boldsymbol{y}
\label{spiny}
\end{align}
Introduce the function $G: \mathbb{R}^n \times \mathbb{R}^n \to \mathbb{R}$ 
\begin{align*}
G(\boldsymbol{x}, \boldsymbol{y}) &= 
   U(\boldsymbol{y}) - U(\boldsymbol{x}) 
 - \frac{1}{2} \left\langle \nabla U(\boldsymbol{y}) + \nabla U(\boldsymbol{x}), \boldsymbol{y}-\boldsymbol{x} \right\rangle \\
+& \frac{h}{4} \left( \left| \nabla U(\boldsymbol{y}) \right|^2 - \left| \nabla U(\boldsymbol{x}) \right|^2 \right)  \text{.}
\end{align*}
Using the expressions for $q_h$ and $\pi$ (cf.~\eqref{ULAqh} and \eqref{pi}, resp.) 
one can show
\[
 \frac{q_h(\boldsymbol{y}, \boldsymbol{x}) \pi(\boldsymbol{y})}
 	{q_h(\boldsymbol{x}, \boldsymbol{y}) \pi(\boldsymbol{x})} = 
	\exp\left(- \beta G(\boldsymbol{x}, \boldsymbol{y}) \right) \text{.}
\]
Hence, \eqref{spiny} can be written as:
\begin{align} 
 & \E^{\boldsymbol{x}}  \left\{ (\alpha_h(\boldsymbol{x},\boldsymbol{X}^*_1) - 1 )^{2 \ell} \right\}  =  \nonumber \\
& \qquad    (4 \pi \beta^{-1} h)^{-n/2}  \int_{R_h(\boldsymbol{x})}  \left(  e^{- \beta G(\boldsymbol{x}, \boldsymbol{y}) } - 1 \right)^{2 \ell}  
e^{\left(  -  \frac{\left| \boldsymbol{y} - \boldsymbol{x} + h  \nabla U(\boldsymbol{x})  \right|^2}{ 4 \beta^{-1} h} \right) } d \boldsymbol{y} \label{spiny2}
\end{align}
Introduce the map $\boldsymbol{\varphi}: \mathbb{R}^n \to \mathbb{R}^n$ 
\begin{align*}
\boldsymbol{\varphi}(\boldsymbol{\xi}) = \boldsymbol{x} 
+ \sqrt{2 h \beta^{-1}} \boldsymbol{\xi}  + h \nabla U(\boldsymbol{x})  \text{.}
\end{align*}
Set $\Tilde{R}_h(\boldsymbol{x}) = \boldsymbol{\varphi}^{-1}( R_h(\boldsymbol{x}) )$.  A change of variables 
of \eqref{spiny2} under $\boldsymbol{\varphi}$ yields,
\begin{align}
 \E^{\boldsymbol{x}} &  \left\{   ( \alpha_h(\boldsymbol{x},\boldsymbol{X}^*_1) - 1 )^{2 \ell}  \right\}   \nonumber \\
&= (2 \pi)^{-n/2}  \int_{\Tilde{R}_h(\boldsymbol{x})} 
\left(  e^{- \beta G(\boldsymbol{x}, \boldsymbol{x} + \sqrt{2 h \beta^{-1}} \boldsymbol{\xi} + h \nabla U(\boldsymbol{x})) } - 1 \right)^{2 \ell} 
e^{ -  \frac{1}{2} \left|  \boldsymbol{\xi}  \right|^2 } d \boldsymbol{\xi}  
\label{spinxi}
\end{align}

For $h$ sufficiently small the latter integral can be well-approximated by a 
Taylor expansion as follows.  A Taylor expansion of the function $G$ about $h=0$ yields,
\begin{align*}
& G(\boldsymbol{x}, \boldsymbol{x} + \sqrt{2 h \beta^{-1}} \boldsymbol{\xi}  + h \nabla U(\boldsymbol{x})) = \\
&  \qquad \left( \frac{\sqrt{2 \beta^{-1}}}{2} \left\langle \nabla U(\boldsymbol{x}), D^2  U(\boldsymbol{x}) \cdot \boldsymbol{\xi} \right\rangle 
- \frac{2 \beta^{-3/2} }{3} D^3 U(\boldsymbol{x}) \cdot \boldsymbol{\xi}^3 \right) h^{3/2}
+ \mathcal{O}(h^2)
\end{align*}
Substitute this expansion into \eqref{spinxi} to obtain,
\begin{align*}
 &  \int_{\Tilde{R}_h(\boldsymbol{x})}  
  \left(  e^{- \beta G(\boldsymbol{x}, \boldsymbol{x} + \sqrt{2 h \beta^{-1}} \boldsymbol{\xi} + h \nabla U(\boldsymbol{x}))} - 1 \right)^{2 \ell} 
  e^{  -  \frac{1}{2}  \left|  \boldsymbol{\xi}  \right|^2 } d \boldsymbol{\xi} =   \\
&  \quad  \int_{\Tilde{R}_h(\boldsymbol{x})}  \left(  h^{3 \ell}
  \left(\frac{\sqrt{2 \beta^{-1}}}{2} \left\langle \nabla U(\boldsymbol{x}), D^2 U(\boldsymbol{x}) \cdot
      \boldsymbol{\xi} \right\rangle  - \frac{2  \beta^{-3/2}}{3} D^3 U(\boldsymbol{x}) \cdot \boldsymbol{\xi}^3 \right)^{2 \ell} 
      + \mathcal{O}(h^{3 \ell + 1}) \right) e^{ - \frac{1}{2} | \boldsymbol{\xi} |^2}  d \boldsymbol{\xi} 
\end{align*}
An application of Laplace method to approximately evaluate integrals yields,
\begin{align*}
 &   \int_{\Tilde{R}_h(\boldsymbol{x})}  
  \left(  e^{- \beta G(\boldsymbol{x}, \boldsymbol{x} + \sqrt{2 h \beta^{-1}} \boldsymbol{\xi} + h \nabla U(\boldsymbol{x}))} - 1 \right)^{2 \ell} 
  e^{  -  \frac{\beta}{4}   \left|  \boldsymbol{\xi}  \right|^2 } d \boldsymbol{\xi} =   \\
& \quad   \int_{\mathbb{R}^n}  h^{3 \ell}
  \left(\frac{\sqrt{2 \beta^{-1}}}{2} \left\langle \nabla U(\boldsymbol{x}), D^2 U(\boldsymbol{x}) \cdot
      \boldsymbol{\xi} \right\rangle  -  \frac{2  \beta^{-3/2}}{3} D^3 U(\boldsymbol{x}) \cdot \boldsymbol{\xi}^3 \right)^{2 \ell} e^{ -
      \frac{1}{2} | \boldsymbol{\xi} |^2} d \boldsymbol{\xi} +   \mathcal{O}(h^{3 \ell + 1})
\end{align*}
The proof is completed by invoking Assumption~\ref{sa} (E).
\end{proof}

\begin{lemma}[Local Accuracy of MALA]
Assume \ref{sa}.  For all $h>0$ there exists a $C>0$ such that
\begin{description}
\item[A)] the local mean-squared error of~MALA satisfies
\[
 \E^{\boldsymbol{x}} \left\{ \left| \boldsymbol{X}_1 -  \boldsymbol{Y}(h)  \right|^2 \right\}  
 \le C  (1 + U( \boldsymbol{x})^4 )  h^{5/2} \text{,}~~\forall~ \boldsymbol{x} \in \mathbb{R}^n \text{;}
\]
\item[B)] the local mean deviation of~MALA satisfies
\[
\left| \E^{\boldsymbol{x}} \left\{ \boldsymbol{X}_1 -  \boldsymbol{Y}(h)  \right\} \right| 
\le C  (1 + U( \boldsymbol{x} )^3 ) h^2 \text{,}~~\forall~ \boldsymbol{x} \in \mathbb{R}^n \text{.}
\]
\end{description}
\label{MALAlocalaccuracy}
\end{lemma}

\begin{proof}
\textbf{A)} By definition of MALA,
\begin{align*}
\E^{\boldsymbol{x}} & \left\{    \left| \boldsymbol{X}_1 - \boldsymbol{Y}(h) \right|^2  \right\} = \\
& \E^{\boldsymbol{x}} \left\{ \left| \boldsymbol{X}^*_1 - \boldsymbol{Y}(h) \right|^2 \alpha_h(\boldsymbol{x},\boldsymbol{X}^*_1) 
+ \left| \boldsymbol{x}-\boldsymbol{Y}(h) \right|^2 (1- \alpha_h(\boldsymbol{x},\boldsymbol{X}^*_1))   \right\} \text{.}
\end{align*}
Hence,
\begin{align} \label{MALAsquaredmserror}
& \E^{\boldsymbol{x}}  \left\{   \left| \boldsymbol{X}_1 - \boldsymbol{Y}(h) \right|^2  \right\} \le \nonumber \\
& \qquad \E^{\boldsymbol{x}} \left\{ | \boldsymbol{X}^*_1 - \boldsymbol{Y}(h) |^2 \right\} 
+ \E^{\boldsymbol{x}} \left\{  | \boldsymbol{x} -\boldsymbol{Y}(h) |^2 (1- \alpha_h(\boldsymbol{x},\boldsymbol{X}^*_1))   \right\}  \text{.}
\end{align}
It is clear from this expression that the local mean-squared error of the Metropolis-Hastings 
integrator is due to the local mean-squared error of forward Euler-Maruyama and a term 
due to probable rejections in the Metropolis-Hastings step.  By the Cauchy-Schwarz inequality 
this latter term can be bounded as follows,
\begin{align*}
&\E^{\boldsymbol{x}} \left\{  | \boldsymbol{x}-\boldsymbol{Y}(h) |^2 (\alpha_h(\boldsymbol{x},\boldsymbol{X}^*_1) - 1) \right\} 
  \le \\
& \qquad  \left( \E^{\boldsymbol{x}} \left\{ \left| \boldsymbol{x}-\boldsymbol{Y}(h) \right|^4 \right\} \right)^{1/2} \cdot 
  \left( \E^{\boldsymbol{x}} \left\{ (\alpha_h(\boldsymbol{x},\boldsymbol{X}^*_1) - 1)^2 \right\} \right)^{1/2} \text{.}
\end{align*}
Lemma~\ref{sdemomentbound} implies there exists a constant $K>0$ such that
\[
\left( \E^{\boldsymbol{x}} \left\{ |  \boldsymbol{x} - \boldsymbol{Y}(h)  |^4 \right\} \right)^{1/2} 
\le K ( 1 + U( \boldsymbol{x} )^2 ) h \text{.}
\]
While Lemma~\ref{MALAstagnationprobability}  implies that there exists a constant $K>0$ such that
\[
\left( \E^{\boldsymbol{x}}\left\{ (\alpha_h(\boldsymbol{x},\boldsymbol{X}^*_1) - 1)^2 \right\} \right)^{1/2} 
\le K (1 + U( \boldsymbol{x} )^2 ) h^{3/2}
\]
Thus, there exists a constant $K>0$ such that
\[
 \E^{\boldsymbol{x}} \left\{  | \boldsymbol{x}-\boldsymbol{Y}(h) |^2 (1-\alpha_h(\boldsymbol{x},\boldsymbol{X}^*_1)) \right\}   
 \le K ( 1 + U( \boldsymbol{x} )^4 ) h^{5/2} \text{.}
\]
Observe that \eqref{MALAsquaredmserror} is dominated by the error incurred when a proposal move is rejected, 
and not the $O(h^3)$ error incurred when a proposal move is accepted (cf.~Lemma~\ref{ULAlocalaccuracy}).  
Hence, the local mean-squared accuracy of MALA is $O(h^{5/4})$.

~
\noindent 
\textbf{B)} Similar to the proof of part (A)
\[
\begin{aligned}
  \E^{\boldsymbol{x}} \left\{ \boldsymbol{X}_1 -
    \boldsymbol{Y}(h) \right\} & = \E^{\boldsymbol{x}} \{
    (\boldsymbol{X}_1^* - \boldsymbol{Y}(h) )
    \alpha_h(\boldsymbol{x},\boldsymbol{X}_1^*)\\
    &\qquad + (\boldsymbol{x}-\boldsymbol{Y}(h) ) (1-
    \alpha_h(\boldsymbol{x},\boldsymbol{X}_1^*)) \}
\end{aligned}
\]
By the triangle inequality 
\[
\begin{aligned}
 \left| \E^{\boldsymbol{x}}  \left\{ \boldsymbol{X}_1 -
    \boldsymbol{Y}(h) \right\} \right| &\le 
     \left| \E^{\boldsymbol{x}} \left\{    (\boldsymbol{X}_1^* - \boldsymbol{Y}(h) ) \alpha_h(\boldsymbol{x},\boldsymbol{X}_1^*) \right\}  \right|  \\
    &\qquad+  \left| \E^{\boldsymbol{x}} \left\{ (\boldsymbol{x}-\boldsymbol{Y}(h) ) (1-
    \alpha_h(\boldsymbol{x},\boldsymbol{X}_1^*)) \right\} \right|  \\
    &\le \left| \E^{\boldsymbol{x}} \left\{
    (\boldsymbol{X}_1^* - \boldsymbol{Y}(h) ) \right\} \right| \\
    &\qquad+   \left| \E^{\boldsymbol{x}} \left\{ (\boldsymbol{x}-\boldsymbol{Y}(h) ) (1-
    \alpha_h(\boldsymbol{x},\boldsymbol{X}_1^*)) \right\} \right|  
\end{aligned}
\]
It is clear from this expression that the local mean deviation of the Metropolis-Hastings 
integrator is due to the local mean-deviation of forward Euler-Maruyama and a term 
due to probable rejections in the Metropolis-Hastings step.  By the Cauchy-Schwarz 
and Jensen inequalities,
\begin{align*}
  \left| \E^{\boldsymbol{x}} \left\{    
      \boldsymbol{X}_1 - \boldsymbol{Y}(h)  \right\} \right|^2 
  &\le 2 \left| \E^{\boldsymbol{x}} \left\{ (\boldsymbol{X}^*_1-\boldsymbol{Y}(h))  \right\} \right|^2  \\
  &\qquad + 2 \left| \E^{\boldsymbol{x}} \left\{ (\boldsymbol{x}-\boldsymbol{Y}(h) ) (1- \alpha_h(\boldsymbol{x},\boldsymbol{X}_1^*)) \right\} \right|^2  \\
  \le& 2 \left| \E^{\boldsymbol{x}} \left\{ (\boldsymbol{X}^*_1-\boldsymbol{Y}(h))  \right\} \right|^2 \\
  & +  2 \left( \E^{\boldsymbol{x}} \left\{ \left| \boldsymbol{x}-\boldsymbol{Y}(h) \right|^4 \right\} \right)^{1/2} 
  \cdot  \left( \E^{\boldsymbol{x}}\left\{ (\alpha_h(\boldsymbol{x},\boldsymbol{X}^*_1) - 1)^4 \right\} \right)^{1/2}
\end{align*}
As in part (A) of this proof, the local mean deviation of forward Euler-Maruyama 
Lemma~\ref{ULAlocalaccuracy} together with Lemma~\ref{sdemomentbound} and 
Lemma~\ref{MALAstagnationprobability} imply that,
\[
\left| \E^{\boldsymbol{x}} \left\{    \boldsymbol{X}_1 -\boldsymbol{Y}(h)   \right\} \right|^2 
\le K ( 1 + F(\boldsymbol{x})^{6} ) h^4 \text{.}
\]
\end{proof}

\begin{lemma}[Uniform In Time Bound on Error]
Assume~\ref{sa}.  Then for all $h>0$ and for all $t>0$, there exists a real constant $C>0$ such that
\[
 \E_{\mu}  \E^{\boldsymbol{x}} \left\{   \left| \boldsymbol{X}_{\lfloor t/h \rfloor}  - \boldsymbol{Y}(t)  \right|^2  \right\} 
 \le C \text{.}
\]
\label{MALAboundonaccuracy}
\end{lemma}

\begin{proof}
This bound is a consequence of ergodicity of the solution to \eqref{SDE1} and MALA
with respect to the probability measure $\mu$.  That is, the ergodic property of the 
integrator and the solution of \eqref{SDE1} imply that
\begin{align*}
 & \E_{\mu}  \E^{\boldsymbol{x}} \left\{   \left| \boldsymbol{X}_{\lfloor t/h \rfloor}  - \boldsymbol{Y}(t)  \right|^2    \right\} \\
& \qquad  \le  
2  \E_{\mu} \E^{\boldsymbol{x}} \left\{   \left| \boldsymbol{X}_{\lfloor t/h \rfloor} \right|^2 \right\} + 
2  \E_{\mu} \E^{\boldsymbol{x}} \left\{   \left| \boldsymbol{Y}(t)  \right|^2  \right\} \\
& \qquad = 4 \int_{\mathbb{R}^n} | \boldsymbol{x} |^2  \mu(d \boldsymbol{x}) \text{.}
\end{align*}
\end{proof}

We are now in measure to prove Theorem~\ref{MALAaccuracy} which we restate.

\begin{thm21}[MALA Equilibrium Strong Accuracy]
Assume~\ref{sa}. For all $T>0$ there exist $h_c > 0$ and $C(T)>0$, such that for 
all positive $h<h_c$ and for all $t \in [0, T]$,
\[
\left( \E_{\mu} \E^{\boldsymbol{x}} \left\{   
\left| \boldsymbol{X}_{\lfloor t/h \rfloor} - \boldsymbol{Y}(t)  \right|^2  \right\} \right)^{1/2} 
\le C(T) h^{3/4} \text{.}
\]
\end{thm21}

\begin{proof}
Discretize the interval $[0, T]$ using $N$ equally spaced points in such a fashion 
that $t_k = k h$ for $k=0,...,N$ with $N=\lfloor T/h \rfloor$.  Assume $h>0$ but otherwise 
arbitrary for now.   The goal in this proof is to obtain a global error estimate for the
mean-squared error of MALA by writing the error at the $(k+1)$th step in terms of the 
error at the $k$th step.   For clarity of presentation we will use a rolling constant 
$K$ throughout this proof.  Set
\[
\epsilon_{k+1} = 
\E_{\mu} \E^{\boldsymbol{x}}  \left\{
 \left| \boldsymbol{X}_{k+1} - \boldsymbol{Y}_{t_{k+1},0}(\boldsymbol{x})  \right|^2 \right\} 
\]
for $k=0,...,N-1$.  Specifically, the proof shows that there exists a constant $K>0$ such that
\begin{align}
\epsilon_{k+1}  \le (1+A h) \epsilon_k + K h^{5/2} \text{.}
\label{globalerrorrecursionrelation}
\end{align}
By unraveling this recursive inequality, the order $h^{3/4}$ mean-squared accuracy 
becomes apparent.  To do this the expectation of the error at the $(k+1)th$ step is 
conditioned on knowing all events up to the $k$th step.  For this purpose let $\mathcal{F}_k$ 
denote the sigma-algebra of events up to and including the $kth$-iteration.
Also write,
\begin{align*}
&\left| \boldsymbol{X}_{k+1}  - \boldsymbol{Y}_{t_{k+1},0}(\boldsymbol{x})  \right|^2 = \nonumber \\
& \qquad \left| \boldsymbol{X}_{k+1} - \boldsymbol{Y}_{t_{k+1},t_k}(\boldsymbol{X}_k)  
+ \boldsymbol{Y}_{t_{k+1},t_k}(\boldsymbol{X}_k) - \boldsymbol{Y}_{t_{k+1},t_k}(\boldsymbol{Y}_{t_k,0}(\boldsymbol{x}))  \right|^2 
\end{align*}
Expanding this expression gives:
\begin{align} \label{totalmserror}
&\left| \boldsymbol{X}_{k+1}  - \boldsymbol{Y}_{t_{k+1},0}(\boldsymbol{x})  \right|^2 = \nonumber \\
& \qquad \left| \boldsymbol{X}_{k+1} - \boldsymbol{Y}_{t_{k+1},t_k}(\boldsymbol{X}_k)  \right|^2  
+  \left| \boldsymbol{Y}_{t_{k+1},t_k}(\boldsymbol{X}_k) - \boldsymbol{Y}_{t_{k+1},t_k}(\boldsymbol{Y}_{t_k,0}(\boldsymbol{x}))  \right|^2 \nonumber \\
& \qquad +  2 \left\langle \boldsymbol{X}_{k+1} - \boldsymbol{Y}_{t_{k+1},t_k}(\boldsymbol{X}_k),  
\boldsymbol{Y}_{t_{k+1},t_k}(\boldsymbol{X}_k) - \boldsymbol{Y}_{t_{k+1},t_k}(\boldsymbol{Y}_{t_k,0}(\boldsymbol{x}))  \right\rangle 
\end{align}
It follows from the local mean-squared accuracy of MALA (cf.~Lemma~\ref{MALAlocalaccuracy}),
\begin{align*}
 \E \{   | \boldsymbol{X}_{k+1} - \boldsymbol{Y}_{t_{k+1},t_k}(\boldsymbol{X}_k)  |^2  
~|~ \mathcal{F}_k  \} \le K  \left(1 +  U(\boldsymbol{X}_k)^4  \right) h^{5/2}  \text{.}
\end{align*}
Since MALA preserves $\mu$ by design, and by the law of total expectation,
\begin{align*}
\E_{\mu} \E^{\boldsymbol{x}}  \{  \E \{   | \boldsymbol{X}_{k+1} - \boldsymbol{Y}_{t_{k+1},t_k}(\boldsymbol{X}_k)  |^2 ~|~ \mathcal{F}_k  \}  \} 
\le K  \left(1 +  \mu(U^4)  \right) h^{5/2}  
\end{align*}
where $\mu(U^4) = \int_{\mathbb{R}^n} U^4 d\mu$.  Similarly, Lemma~\ref{regularityofsolutions} 
implies that
\begin{align}
& \E_{\mu}  \E^{\boldsymbol{x}}  \left\{
\left| \boldsymbol{Y}_{t_{k+1},t_k}(\boldsymbol{X}_k) - \boldsymbol{Y}_{t_{k+1},t_k}(\boldsymbol{Y}_{t_k,0}(\boldsymbol{x}))  \right|^2   
\right\} \nonumber \le ( 1 + K h )  \epsilon_k
\end{align}
This leaves the third term in \eqref{totalmserror}.  Set 
\[
\boldsymbol{\Delta} =  \boldsymbol{Y}_{t_{k+1},t_k}(\boldsymbol{X}_k) - \boldsymbol{Y}_{t_{k+1},t_k}(\boldsymbol{Y}_{t_k,0}(\boldsymbol{x})) 
- ( \boldsymbol{X}_k - \boldsymbol{Y}_{t_k,0}(\boldsymbol{x}) )
\]
In terms of which, rewrite the third term in \eqref{totalmserror} as:
\begin{align*}
 \left\langle \boldsymbol{X}_{k+1} - \boldsymbol{Y}_{t_{k+1},t_k}(\boldsymbol{X}_k),  \boldsymbol{X}_k - \boldsymbol{Y}_{t_k,0}(\boldsymbol{x})  \right\rangle   
 +   \left\langle \boldsymbol{X}_{k+1} - \boldsymbol{Y}_{t_{k+1},t_k}(\boldsymbol{X}_k), \boldsymbol{\Delta}  \right\rangle  \text{.}
\end{align*}
Cauchy-Schwarz inequality implies,
\begin{align*}
& \E_{\mu} \E^{\boldsymbol{x}} \left\{  \left\langle\E \left\{  \boldsymbol{X}_{k+1} - \boldsymbol{Y}_{t_{k+1},t_k}(\boldsymbol{X}_k) 
~ |~ \mathcal{F}_k \right\},  \boldsymbol{X}_k - \boldsymbol{Y}_{t_k,0}(\boldsymbol{x})  \right\rangle \right\} \\
&\quad
\le  \E_{\mu} \E^{\boldsymbol{x}} \left\{ \left| \E \left\{  \boldsymbol{X}_{k+1} - \boldsymbol{Y}_{t_{k+1},t_k}(\boldsymbol{X}_k) 
~ |~ \mathcal{F}_k \right\} \right|^2 \right\} ^{1/2}   \epsilon_k^{1/2}
\end{align*}
It follows from Lemma~\ref{MALAlocalaccuracy} and invariance of $\mu$ under MALA that,
\begin{align*}
\E_{\mu} \E^{\boldsymbol{x}} \left\{ \left| \E \left\{  \boldsymbol{X}_{k+1} - \boldsymbol{Y}_{t_{k+1},t_k}(\boldsymbol{X}_k) 
~ |~ \mathcal{F}_k  \right\} \right|^2 \right\} ^{1/2}    \le K    \left(1 +  \mu(U^3)  \right) h^{2} 
\end{align*}
Cauchy-Schwarz inequality also implies,
\begin{align*}
 & \E_{\mu}  \E^{\boldsymbol{x}}   \left\{  \E \left\{  \left\langle \boldsymbol{X}_{k+1} - \boldsymbol{Y}_{t_{k+1},t_k}(\boldsymbol{X}_k), \boldsymbol{\Delta}  \right\rangle  
 ~|~ \mathcal{F}_k   \right\}  \right\} \\
 & \quad \le  \E_{\mu} \E^{\boldsymbol{x}}   \left\{  \E \left\{ \left|  \boldsymbol{X}_{k+1} - \boldsymbol{Y}_{t_{k+1},t_k}(\boldsymbol{X}_k) \right|^2 
 ~|~ \mathcal{F}_k  \right\} \right\} ^{1/2}
 \E_{\mu}  \E^{\boldsymbol{x}}   \left\{  \E \left\{ \left| \boldsymbol{\Delta} \right|^2 
 ~|~ \mathcal{F}_k \right\} \right\} ^{1/2}  \\
 & \quad \le K    \left(1 +  \mu(U^2)  \right)   \E_{\mu}  \E^{\boldsymbol{x}}   \left\{  \E \left\{ \left| \boldsymbol{\Delta} \right|^2 
 ~|~ \mathcal{F}_k   \right\}  \right\} ^{1/2}  h^{5/4}
\end{align*}
However, according to Lemma~\ref{regularityofsolutions},
\begin{align*}
\E_{\mu}  \E^{\boldsymbol{x}}   \left\{  \E \left\{ \left| \boldsymbol{\Delta} \right|^2 
 ~|~ \mathcal{F}_k   \right\}  \right\}^{1/2} \le K h ( 1 + \mu(U) )   \epsilon_k^{1/4}
\end{align*}
In sum, there exists a real constant $K>0$ such that the following term bounds from above 
the third term in \eqref{totalmserror}:
\begin{align*}
K h^{9/4} \epsilon_k^{1/4} + K  h^{2}  \epsilon_k^{1/2} \text{.}
\end{align*}
Using elementary inequalities it is clear that,
\[
2 h^{2}  K \epsilon_k^{1/2} \le h \epsilon_k + h^{5/2} K^2
\]
and
\[
2 h^{9/4} K \epsilon_k^{1/4} \le h^{5/2} K^2 
+ \epsilon_k^{1/2} h^2 \le h^{5/2} K^2 
+ \frac{1}{2} \left( \epsilon_k h + h^{5/2} \right) \text{.}
\]
Combining the bounds on the iterated expectations of \eqref{totalmserror} yields 
\eqref{globalerrorrecursionrelation}.
\end{proof}


\subsection{MALTA and its Properties}

A necessary condition for geometric ergodicity of a Metropolis-Hastings method is that 
the rejection or stagnation probability is bounded away from unity 
(see proposition 5 of \cite{RoTw1996B}).  When the drift in \eqref{ULA} is nonglobally Lipschitz, 
ULA is transient and MALA's stagnation probability cannot be globally bounded away from unity.   
Hence, MALA is not geometrically ergodic.  MALTA corrects this problem by truncating the drift 
in the candidate dynamics in regions where the Euler-Maruyama integrator can be explosive.   
This modification enables a precise control of the moments of the integrator initiated from a 
nonequilibrium initial condition.  With this control we are able to prove pathwise accuracy of 
MALTA on finite time-intervals.

The geometric ergodicity of MALTA for super-exponential target densities with non-degenerate 
level-sets was intuitively clear to the investigators who proposed the algorithm in\cite{RoTw1996A}.  
Subsequently, it was proven in proposition 2.1 of\cite{At2005} using techniques to prove 
geometric ergodicity for  random walk Metropolis candidate dynamics given in\cite{JaHa2000}.  
For the convenience of the reader, the theorem is restated.

\begin{theorem}[MALTA Geometric Ergodicity]
Assume \ref{sa}.  For every $h>0$, there exists a nonnegative function 
$M: \mathbb{R}^n \to \mathbb{R}$, real constant $\rho<1$, such that
the $k$-step transition probability $ P_h^k(\boldsymbol{x}, \cdot)$ of MALTA converges 
to $\mu$ geometrically fast:
\[
\| P_h^k(\boldsymbol{x}, \cdot) - \mu \|_{M} \le M( \boldsymbol{x}) \rho^{k},
~~\forall~k \in \mathbb{N},~~\forall~\boldsymbol{x} \in \mathbb{R}^n \text{,}
\]
where we have introduced a TV norm $\| \cdot \|_{M}$ which is defined for a signed 
measure $\nu$ as:
\[
\| \nu  \|_{M} = \sup_{f : ~|f| \le M} | \nu(f)  | \text{.}
\]
\end{theorem}

The main implication of geometric ergodicity for our purpose is the following crucial estimate 
on moments of MALTA.  The proof of this lemma is a straightforward, but tedious consequence 
of the results in Proposition 2.1 and Lemma 6.2 of\cite{At2005}.

\begin{lemma}[Estimates on Higher Moments of MALTA]
Assume \ref{sa}.   For every $E_0>0$, for every $\ell \ge 1$, there exists a $h_c>0$
and $C(E_0)>0$  such that for all positive $h<h_c$, for all 
$\boldsymbol{x} : U(\boldsymbol{x}) \le E_0$, and for all $t>0$,
\[
\sup_{h < h_c} \E^{\boldsymbol{x}} \left\{   U(\boldsymbol{Z}_{\lfloor t/h \rfloor})^{\ell} \right\} 
\le  C(E_0) \text{.}
\]
\label{MALTAmomentbound}
\end{lemma}

We are now in measure to prove Theorem~\ref{MALTAaccuracy}.  
\begin{thm22}[MALTA Strong Accuracy]
Assume \ref{sa}.  Then, for every $E_0>0$ and $T>0$, there exists a $h_c(E_0) > 0$ 
and a $C(T, E_0)>0$ such that for all positive $h<h_c$, 
for all $\boldsymbol{x} : U(\boldsymbol{x}) \le E_0$, and for all $t \in [0, T]$,
\[
\left( \E^{\boldsymbol{x}} \left\{  
 \left| \boldsymbol{Z}_{\lfloor t/h \rfloor}  - \boldsymbol{Y}_{t,0}  \right|^2  \right\}  \right)^{1/2} 
\le C(T, E_0) h^{3/4} \text{.}
\]
\end{thm22}

\begin{proof}
The proof of this theorem is similar to the proof of Theorem~\ref{MALAaccuracy} with the 
following main differences.
\begin{itemize}
\item The preservation of the moments of the solution to \eqref{SDE1} and the 
Metropolis-Hastings method due to ergodicity are replaced by bounds on moments 
implied by Lemmas~\ref{sdemomentbound} and \ref{MALTAmomentbound}.
\item The $k+1$-step error conditioned on knowing $\mathcal{F}_k$ is split into two parts:
\begin{align*}
\E&\left\{   \left| \boldsymbol{Z}_{k+1}  - \boldsymbol{Y}(t_{k+1}) \right|^2  
					~|~ \mathcal{F}_k \right\}   \le  \\
&\E\left\{   \left| \boldsymbol{Z}_{k+1}  - \boldsymbol{Y}(t_{k+1})  \right|^2  
					~|~ \mathcal{F}_k ~\&~ h | \nabla U(  \boldsymbol{Z}_k ) | \le 1\right\} 
					P(h | \nabla U(  \boldsymbol{Z}_k ) | \le 1 ) \\
+& \E \left\{   \left| \boldsymbol{Z}_{k+1}  - \boldsymbol{Y}(t_{k+1})  \right|^2 
					~|~ \mathcal{F}_k ~\&~  h | \nabla U(  \boldsymbol{Z}_k ) | > 1 \right\}   
					P(h |  \nabla U(  \boldsymbol{Z}_k ) | > 1 ) 
\end{align*}
For the first part, the results of Theorem~\ref{MALAaccuracy} apply with the provision given 
in the first difference above.   For the second term, Chebyshev's inequality implies:
\[
P(h | \nabla U(  \boldsymbol{Z}_k ) | > 1 )  
\le h^4 \E^{\boldsymbol{x}} \left\{ | \nabla U(  \boldsymbol{Z}_k ) |^4  \right\}  \text{.}
\]
This inequality, Assumption \ref{sa} (E), Lemma~\ref{sdemomentbound}, and 
Lemma~\ref{MALTAmomentbound} imply that there exists a constant $C(E_0)>0$ such that,
\[
\E^{\boldsymbol{x}} \left\{   \left| \boldsymbol{Z}_{k+1}  - \boldsymbol{Y}(t_{k+1})  \right|^2 \right\}  
P(h | \nabla U(  \boldsymbol{Z}_k ) | > 1 )  \le h^2 C(E_0) \text{.}
\]
\end{itemize}
\end{proof}

\newpage


\section{Inertial Langevin} \label{sec:InertialLangevin}

Next we shall consider a Metropolized version of a discretization 
of \eqref{SDE2} that extends variational integrators to 
inertial Langevin dynamics.    We begin by introducing the so-called geometric 
Langevin algorithm and discuss its properties.   The section then examines the 
properties of the Metropolis-Hastings adjusted geometric Langevin algorithm, 
including quantifying its accuracy in approximating inertial Langevin dynamics.

The arguments in this section are very closely related to those in \S \ref{sec:OverdampedLangevin}.
However, there are two key differences.  First, the solution to \eqref{SDE2} is no longer
a reversible stochastic process.  Yet, the transition
kernel of the solution process composed with a momentum flip does satisfy detailed balance.
Second, the diffusion in \eqref{SDE2} is only applied to momentum and not position.   
These differences motivate this section on Metropolizing discretizations of nonreversible
processes.  However, we will omit analysis that is redundant.


\subsection{Geometric Langevin Algorithm}

Let $N$ and $h$ be given, set $T=  N h$ and $t_k = h k $ for
$k=0,...,N$.   We shall consider an integrator for \eqref{SDE2} based on 
splitting the Langevin equations into Hamilton's equations for the Hamiltonian
$H$ \eqref{HamiltonsEquations} and  and Ornstein-Uhlenbeck equations 
\eqref{OrnsteinUhlenbeck}.  The solution of Hamilton's equations will be 
approximated by the discrete Hamiltonian map of a variational integrator\cite{MaWe2001}.  
While the exact flow will be used for the Ornstein-Uhlenbeck equations.   These 
flows will be composed in a Strang-type fashion to obtain a pathwise approximant 
to Langevin's equations which we will call the Geometric Langevin Algorithm (GLA).

\paragraph{Variational Integrators.}

Let $L: \mathbb{R}^{2n} \to \mathbb{R}$ denote the Lagrangian obtained from the 
Legendre transform of the Hamiltonian $H$, and given by:
\[
L(\boldsymbol{q}, \boldsymbol{v}) = 
\frac{1}{2} \boldsymbol{v}^T \boldsymbol{M}  \boldsymbol{v} - U(\boldsymbol{q}) \text{.}
\]
A variational integrator is defined by a discrete Lagrangian $L_d : \mathbb{R}^n \times \mathbb{R}^n 
\times \mathbb{R}^+ \to \mathbb{R}$ which is an approximation to the so-called {\em exact discrete 
Lagrangian} which is defined as:
\[
L_d^E(\boldsymbol{q}_0, \boldsymbol{q}_1, h) = \int_0^h L( \boldsymbol{Q}, \dot{\boldsymbol{Q}} ) dt
\]
where $\boldsymbol{Q}(t)$ solves the Euler-Lagrange equations for the Lagrangian $L$ 
with endpoint  conditions $\boldsymbol{Q}(0) = \boldsymbol{q}_0$ and 
$\boldsymbol{Q}(h) = \boldsymbol{q}_1$.

A discrete Lagrangian determines a symplectic integrator as follows.  
Given $(\boldsymbol{q}_0,\boldsymbol{p}_0) \in \mathbb{R}^{2n}$, 
a variational integrator defines an update 
$(\boldsymbol{q}_1, \boldsymbol{p}_1) \in \mathbb{R}^{2n}$  
by the following system of equations:
\begin{equation} \label{del}
\begin{cases}
\begin{array}{rcl}
\boldsymbol{p}_0 &=& -D_1 L_d(\boldsymbol{q}_0, \boldsymbol{q}_1,h) \text{,} \\
\boldsymbol{p}_1 &=& D_2 L_d(\boldsymbol{q}_0, \boldsymbol{q}_1,h)  \text{.}
\end{array}
\end{cases}
\end{equation}
Denote this map by $\theta_h : \mathbb{R}^{2n} \to  \mathbb{R}^{2n}$, i.e., 
\[
\theta_h: ~~ (\boldsymbol{q}_0,\boldsymbol{p}_0) \mapsto (\boldsymbol{q}_1, \boldsymbol{p}_1)  \text{,}
\] 
where $(\boldsymbol{q}_1, \boldsymbol{p}_1)$ solve \eqref{del}.
One can show that $\theta_h$ preserves the canonical symplectic form 
on $ \mathbb{R}^{2n}$, and hence, is Lebesgue measure preserving \cite{MaWe2001}.
By appropriately constructing $L_d$, the map $\theta_h$ can define an approximation 
to the flow of Hamilton's equations for the Hamiltonian $H$  \eqref{HamiltonsEquations}.

A discrete Lagrangian is self-adjoint if:
\begin{equation} \label{SelfAdjointLd}
L_d(\boldsymbol{q}_0, \boldsymbol{q}_1, h) = L_d(\boldsymbol{q}_1, \boldsymbol{q}_0, h) \text{.}
\end{equation}
Some of the results that follow will be specific to the St\"{o}rmer-Verlet integrator
which can be derived from the following discrete Lagrangian:
\begin{equation} \label{VerletLd}
L_d(\boldsymbol{q}_0, \boldsymbol{q}_1, h) = 
\frac{1}{2h} (\boldsymbol{q}_1 - \boldsymbol{q}_0)^T \boldsymbol{M}  (\boldsymbol{q}_1 - \boldsymbol{q}_0) 
- \frac{h}{2} (U(\boldsymbol{q}_0) + U(\boldsymbol{q}_1) ) \text{.}
\end{equation}
This discrete Lagrangian is clearly self-adjoint.

\paragraph{Ornstein-Uhlenbeck Equations.}

The following stochastic evolution map $\psi_{t_k+h, t_k} :  \mathbb{R}^{2n} \to \mathbb{R}^{2n}$ 
defines the stochastic flow of \eqref{OrnsteinUhlenbeck}:
\begin{align} 
& \psi_{t_k+h,t_k}:  \nonumber \\
& \qquad (\boldsymbol{q},\boldsymbol{p}) \mapsto  \left(\boldsymbol{q},  e^{-\gamma \boldsymbol{M}^{-1} h} \boldsymbol{p} 
+ \sqrt{2 \beta^{-1} \gamma} \int_{t_k}^{t_k+h} e^{-\gamma \boldsymbol{M}^{-1} (t_k+h-s)} d \boldsymbol{W}(s) \right) \text{.}  
\label{eq:psi}
\end{align}
For the distribution of the solution, the stochastic flow will be denoted simply by $\psi_{h}$. 
Let $o_h$ denote the transition probability density of $\psi_{t_k+h,t_k}$.  By a change of variables, 
it's transition density is given explicitly by:
\begin{align*}
o_h&  (\boldsymbol{p}_0, \boldsymbol{p}_1) = \\    
& \frac{1}{( 2 \pi )^{n/2}  | \det( \boldsymbol{\Sigma}_h) | }  
\exp\left(-\frac{1}{2} \left( \boldsymbol{p}_1 - e^{- \gamma \boldsymbol{M}^{-1} h} \boldsymbol{p}_0 \right)^T  
\boldsymbol{\Sigma}_h^{-1} \left( \boldsymbol{p}_1 - e^{- \gamma \boldsymbol{M}^{-1} h} \boldsymbol{p}_0 \right) \right)  
\text{,}
\end{align*}
where
\[
 \boldsymbol{\Sigma}_h = \beta^{-1} \left( \boldsymbol{Id}  - \exp(- 2 \gamma \boldsymbol{M}^{-1} h) \right) \boldsymbol{M}  \text{.}
\]
Observe that this transition density does not depend on the initial or terminal position, since \eqref{eq:psi}
fixes position and the Hamiltonian is separable.

\paragraph{Strang-type Splitting.}

GLA is defined as the following Strang-type splitting pathwise approximant to \eqref{SDE2}:
\begin{equation} \label{GLA}
\Tilde{\boldsymbol{X}}_{k+1} := (\Tilde{\boldsymbol{Q}}_{k+1},
\Tilde{\boldsymbol{P}}_{k+1} ) = \psi_{t_k+h,t_k+h/2} \circ \theta_h \circ \psi_{t_k+h/2,t_k}(\Tilde{\boldsymbol{Q}}_{k},
\Tilde{\boldsymbol{P}}_{k} ) 
\end{equation}
for $k=0,...,N-1$.  For any $A \in \mathcal{B}(\mathbb{R}^{2n})$, the transition probability kernel for GLA is given by:
\begin{align*}
Q_h  ((\boldsymbol{q}_0,\boldsymbol{p}_0), A)  = \int_{\mathbb{R}^{2 n} \times A } 
o_{h/2}(\boldsymbol{p}_0 , \boldsymbol{p}_{0}^*) \cdot o_{h/2}(\boldsymbol{p}_1^*,\boldsymbol{p}_1) 
\cdot \delta_{\theta_h(\boldsymbol{q}_0, \boldsymbol{p}_0^*)} 
( d \boldsymbol{q}_1, d \boldsymbol{p}_1^* )   d \boldsymbol{p}_0^* d \boldsymbol{p}_1
\end{align*}
Observe that the zero of the Dirac-delta measure is implicitly defined by
\begin{equation*}
\begin{cases}
\begin{array}{rcl}
\boldsymbol{p}_0^* &=& -D_1 L_d(\boldsymbol{q}_0, \boldsymbol{q}_1,h) \text{,} \\
\boldsymbol{p}_1^* &=& D_2 L_d(\boldsymbol{q}_0, \boldsymbol{q}_1,h)  \text{.}
\end{array}
\end{cases}
\end{equation*}
To make it explicit, we perform a change of variables,
\begin{align*}
Q_h  ((\boldsymbol{q}_0,\boldsymbol{p}_0), A) & = \int_{\mathbb{R}^{2 n} \times A }  
o_{h/2}(\boldsymbol{p}_0 , \boldsymbol{p}_{0}^*) \cdot 
o_{h/2}(\boldsymbol{p}_1^*,\boldsymbol{p}_1)    \cdot 
| \det(D_{12} L_d(\boldsymbol{q}_0, \boldsymbol{q}_1,h) )|   \\
& \qquad \cdot \delta_{(-D_1 L_d(\boldsymbol{q}_0, \boldsymbol{q}_1,h), D_2L_d(\boldsymbol{q}_0, \boldsymbol{q}_1,h))} 
( d \boldsymbol{p}_0^*, d \boldsymbol{p}_1^* )   d \boldsymbol{q}_1 d \boldsymbol{p}_1
\end{align*}
From which it follows the transition density $q_h$ of GLA is given by:
\begin{align} 
q&_h   ((\boldsymbol{q}_0,\boldsymbol{p}_0), (\boldsymbol{q}_1,\boldsymbol{p}_1)) = 
| \det(D_{12} L_d(\boldsymbol{q}_0, \boldsymbol{q}_1,h) )|   \nonumber \\
& \cdot o_{h/2}(\boldsymbol{p}_0, -D_1 L_d(\boldsymbol{q}_0, \boldsymbol{q}_1, h))  
\cdot o_{h/2} (D_2 L_d(\boldsymbol{q}_0, \boldsymbol{q}_1, h), \boldsymbol{p}_1)
 \label{GLAdensity} \text{.}
\end{align}
Observe that an explicit characterization of $q_h$ is available 
even when the variational integrator is implicit.

\paragraph{St\"{o}rmer-Verlet Based GLA.}

In this paragraph a local accuracy result is stated for the St\"{o}rmer-Verlet based GLA \eqref{SVGLA}.
Let $\Tilde{\boldsymbol{X}}_1 := (\Tilde{\boldsymbol{Q}}_1, \Tilde{\boldsymbol{P}}_1)$.  Recall,
that given $(\boldsymbol{q}_0 , \boldsymbol{p}_0) \in \mathbb{R}^{2n}$ at time zero and $h>0$, 
the algorithm computes $\Tilde{\boldsymbol{X}}_1$ by the following explicit update rule:
\begin{equation*}
\begin{cases}
& \Tilde{\boldsymbol{Q}}_{1} = \boldsymbol{q}_0 
+ h \boldsymbol{M}^{-1} e^{-\gamma \boldsymbol{M}^{-1} h/2} \boldsymbol{p}_0  
- \frac{h^2}{2} \boldsymbol{M}^{-1} \nabla U(  \boldsymbol{q}_0  ) \\
& \qquad \qquad + h \sqrt{2 \beta^{-1} \gamma} \int_{0}^{h/2} \boldsymbol{M}^{-1} e^{- \gamma \boldsymbol{M}^{-1} (h/2-s) } d\boldsymbol{W}(s) \text{,} \\
& \Tilde{\boldsymbol{P}}_{1} =  e^{-\gamma \boldsymbol{M}^{-1} h}  \boldsymbol{p}_0  
- \frac{h}{2} e^{- \gamma \boldsymbol{M}^{-1} h/2} \left( \nabla U(\boldsymbol{q}_0  ) + \nabla U( \Tilde{\boldsymbol{Q}}_{1} ) \right) \\
& \qquad \qquad + \sqrt{2 \beta^{-1} \gamma} \int_{0}^{h} e^{- \gamma \boldsymbol{M}^{-1} (h -s ) } d \boldsymbol{W}(s) \text{.}
\end{cases}
\end{equation*}
The proof of the following lemma is similar to Lemma~\ref{ULAlocalaccuracy}, 
and hence, is omitted.

\begin{lemma}[Local Accuracy of St\"{o}rmer-Verlet-Based GLA]
Assume \ref{sa} (C) and \ref{sa} (E) on the potential energy.  For all $h>0$ and 
$\boldsymbol{x} = (\boldsymbol{q}_0, \boldsymbol{p}_0) \in \mathbb{R}^{2n}$, there exists 
$C(\boldsymbol{q}_0, \boldsymbol{p}_0)>0$ such that $\mu(C) < \infty$ and
\begin{description}
\item[A)] the local mean-squared error of \eqref{SVGLA} satisfies
\[
\left( \E^{\boldsymbol{x}} \left\{ 
\left| \Tilde{\boldsymbol{X}}_1 -  \boldsymbol{Y}(h)  \right|^2 \right\}  \right)^{1/2} 
\le C(\boldsymbol{q}_0, \boldsymbol{p}_0)  h^{3/2} \text{;}
\]
\item[B)] the local mean deviation of \eqref{SVGLA} satisfies
\[
\left| \E^{\boldsymbol{x}} \left\{ \Tilde{\boldsymbol{X}}_1 -  \boldsymbol{Y}(h)  \right\} \right| 
\le C(\boldsymbol{q}_0, \boldsymbol{p}_0) h^3 \text{.}
\]
\end{description}
\label{GLAlocalaccuracy}
\end{lemma}

As has been demonstrated, the Markov chain  defined by 
sampling \eqref{SVGLA} every time-step possesses a smooth probability transition density 
with respect to Lebesgue measure.  For globally Lipschitz potential forces, this discretization 
is a first-order strongly accurate integrator for \eqref{SDE2}.  However,  if the potential force 
is nonglobally Lipschitz, the St\"{o}rmer-Verlet based GLA is plagued with the 
same transient behavior as Euler-Maruyama for overdamped Langevin.   As before, a 
Metropolis-Hastings method is proposed to stochastically stabilize this Markov chain.


\subsection{MAGLA and its Properties}

MAGLA is the Metropolis-adjusted GLA.  Given $h$ and $(\boldsymbol{Q}_k, \boldsymbol{P}_k)$,
MAGLA computes a proposal move according to a step of GLA:
\[
(\boldsymbol{Q}_{k+1}^*, \boldsymbol{P}_{k+1}^*) = 
\psi_{t_k+h,t_k+h/2} \circ \theta_h \circ \psi_{t_k+h/2,t_k}( \boldsymbol{Q}_k, \boldsymbol{P}_k) \text{.}
\]
MAGLA accepts this proposal move with probability:
\[
\alpha_h((\boldsymbol{q}_0,\boldsymbol{p}_0), (\boldsymbol{q}_1,\boldsymbol{p}_1)) = 1 \wedge 
\frac{q_h ((\boldsymbol{q}_1,\boldsymbol{p}_1), (\boldsymbol{q}_0,-\boldsymbol{p}_0))  \pi(\boldsymbol{q}_1,\boldsymbol{p}_1)}
        {q_h ((\boldsymbol{q}_0,\boldsymbol{p}_0), (\boldsymbol{q}_1,-\boldsymbol{p}_1)) \pi(\boldsymbol{q}_0,\boldsymbol{p}_0)}
\]
Altogether, the MAGLA update is given by:
\begin{align*} 
& \boldsymbol{X}_{k+1} :=  ( \boldsymbol{Q}_{k+1}, \boldsymbol{P}_{k+1} ) = \nonumber \\ 
& \qquad \begin{cases}
(\boldsymbol{Q}_{k+1}^*, \boldsymbol{P}_{k+1}^*)  \qquad &  
\text{if}~~\zeta_k< \alpha_h((\boldsymbol{Q}_k,\boldsymbol{P}_k), (\boldsymbol{Q}_{k+1}^*,\boldsymbol{P}_{k+1}^*)) \\
\varphi(\boldsymbol{Q}_k, \boldsymbol{P}_k) \qquad &  \text{otherwise}
\end{cases} 
\end{align*}
for $k=0,...,N-1$.   We stress the momentum flip in the acceptance probability and
upon rejection is introduced because inertial Langevin dynamics is nonreversible.
Yet, the exact transition density of the solution to inertial Langevin composed with
a momentum is reversible, i.e., does satisfy detailed balance.  We will see in the 
proof of Lemma~\ref{MAGLAstagnationprobability} that without this momentum
flip, pathwise accuracy cannot be achieved with MAGLA.

MAGLA preserves $\mu$, and hence, is not transient even 
when its underlying variational integrator is explicit.   In fact, it is often 
straightforward to classify MAGLA as an ergodic Markov chain even if the 
potential force is nonglobally Lipschitz.   However, with an explicit variational 
integrator MAGLA is often no longer geometrically ergodic even with momentum 
flips.   To correct this problem, a modification of MAGLA can be implemented 
where the drift is truncated in regions where an explicit discretization of the drift 
causes high rejection rates.    This modification which would be analogous to 
MALTA is not implemented here.  Instead, we will concentrate on proving 
pathwise accuracy from an initial condition restricted to the equlibrium measure.  
This restriction enables us to obtain bounds on relevant moments of the 
Metropolized integrator.  To prove pathwise accuracy the following lemmas 
will be useful.

The following lemma shows for an arbitrary symmetric variational integrator, the 
acceptance probability of MAGLA is related to the energy change in its variational 
integrator.  This implies that the acceptance probability of MAGLA has mechanical 
and intrinsic meaning.  It also simplifies the subsequent analysis of MAGLA.

\begin{lemma}[Acceptance Probability of MAGLA]
Let $q_h$ denote the transition probability density of GLA.  
Let $L_d$ denote the discrete Lagrangian associated with the variational integrator $\theta_h$.
Assume that $L_d$ is self-adjoint (cf.~\ref{SelfAdjointLd}).  
Then, the acceptance probability of MAGLA satisfies:
\[
\alpha_h((\boldsymbol{q}_0, \boldsymbol{p}_0), (\boldsymbol{q}_1, \boldsymbol{p}_1)) = 
1 \wedge \exp\left( - \beta \Delta E(\boldsymbol{q}_0, \boldsymbol{q}_1) \right) \text{,}
\]
where we have introduced:
\begin{align} \label{DeltaE}
& \Delta E(\boldsymbol{q}_0, \boldsymbol{q}_1)  =  
     \frac{1}{2} D_2 L_d(\boldsymbol{q}_0, \boldsymbol{q}_1, h)^T \boldsymbol{M}^{-1}  D_2 L_d(\boldsymbol{q}_0, \boldsymbol{q}_1, h) 
  + U( \boldsymbol{q}_1 ) \nonumber \\ 
 & \qquad - \frac{1}{2} D_1 L_d(\boldsymbol{q}_0, \boldsymbol{q}_1, h)^T \boldsymbol{M}^{-1}  D_1 L_d(\boldsymbol{q}_0, \boldsymbol{q}_1, h) 
  - U( \boldsymbol{q}_0 ) \text{.}
\end{align}
\label{MAGLAacceptanceprobability}
\end{lemma}

\begin{proof}
Introduce the abbreviations: $D_i L_d = D_i L_d( \boldsymbol{q}_0, \boldsymbol{q}_1, h)$ for $i=1,2$,
and $\boldsymbol{B}_{h/2} = \exp(-\gamma \boldsymbol{M}^{-1} h/2)$.   Expanding \eqref{GLAdensity} yields,
\begin{align}  \label{q0p0toq1p1}
& q_h   ((\boldsymbol{q}_0,\boldsymbol{p}_0),  (\boldsymbol{q}_1,\boldsymbol{p}_1)  ) =   
\frac{  | \det(D_{12} L_d(\boldsymbol{q}_0, \boldsymbol{q}_1,h) )|  }{ (2 \pi)^n | \det( \boldsymbol{\Sigma}_{h/2}) |}  \nonumber \\
& \quad \exp  \left(  - \frac{\beta}{2}  \left( (D_1 L_d 
+  \boldsymbol{B}_{h/2}  \boldsymbol{p}_0 )^T \boldsymbol{M}^{-1} ( \boldsymbol{Id} -  \boldsymbol{B}_{h/2}^2)^{-1} 
(D_1 L_d +  \boldsymbol{B}_{h/2} \boldsymbol{p}_0 ) \right. \right. \nonumber  \\
& \qquad \qquad \left. \left.+ (  \boldsymbol{p}_1 -  \boldsymbol{B}_{h/2} D_2 L_d )^T \boldsymbol{M}^{-1} 
( \boldsymbol{Id} -  \boldsymbol{B}_{h/2}^2)^{-1} ( \boldsymbol{p}_1 -  \boldsymbol{B}_{h/2} D_2 L_d ) \right) \right)
\end{align}
The self-adjoint property of the discrete Lagrangian implies that:
\begin{align*}
D_2 L_d (  \boldsymbol{q}_0, \boldsymbol{q}_1, h) &= D_1 L_d ( \boldsymbol{q}_1, \boldsymbol{q}_0, h), \\
D_1 L_d (  \boldsymbol{q}_0, \boldsymbol{q}_1, h) &= D_2 L_d ( \boldsymbol{q}_1, \boldsymbol{q}_0, h) \text{.}
\end{align*}
Hence,
\begin{align} \label{q1p1toq0p0}
& q_h   ((\boldsymbol{q}_1,\boldsymbol{p}_1),  (\boldsymbol{q}_0,\boldsymbol{p}_0)  ) =  
\frac{  | \det(D_{12} L_d(\boldsymbol{q}_0, \boldsymbol{q}_1,h) )|  }{ (2 \pi)^n  | \det( \boldsymbol{\Sigma}_{h/2}) |} \nonumber  \\
& \quad \exp  \left(  - \frac{\beta}{2}  \left( 
(D_2 L_d +  \boldsymbol{B}_{h/2}  \boldsymbol{p}_1 )^T \boldsymbol{M}^{-1} 
( \boldsymbol{Id} - \boldsymbol{B}_{h/2}^2)^{-1} (D_2 L_d + \boldsymbol{B}_{h/2} \boldsymbol{p}_1 ) \right. \right. \nonumber \\
& \qquad \qquad \left. \left.+ 
(  \boldsymbol{p}_0 -  \boldsymbol{B}_{h/2} D_1 L_d )^T \boldsymbol{M}^{-1} 
( \boldsymbol{Id} - \boldsymbol{B}_{h/2}^2)^{-1} ( \boldsymbol{p}_0 - \boldsymbol{B}_{h/2} D_1 L_d ) \right) \right)
\end{align}
From \eqref{q0p0toq1p1} and \eqref{q1p1toq0p0}, it follows that
\begin{align*}
& \frac{q_h   ((\boldsymbol{q}_1,\boldsymbol{p}_1),  (\boldsymbol{q}_0,-\boldsymbol{p}_0) )}
{q_h   ((\boldsymbol{q}_0,\boldsymbol{p}_0),  (\boldsymbol{q}_1,-\boldsymbol{p}_1) )}   =  \\
 & \quad \exp   \left(  - \frac{\beta}{2}  \left( 
D_2 L_d ^T \boldsymbol{M}^{-1} D_2 L_d - D_1 L_d ^T \boldsymbol{M}^{-1} D_1 L_d  
- \boldsymbol{p}_1^T \boldsymbol{M}^{-1} \boldsymbol{p}_1 +  \boldsymbol{p}_0^T \boldsymbol{M}^{-1} \boldsymbol{p}_0 \right)  \right)
\end{align*}
And, finally,
\begin{align*}
& \frac{q_h   ((\boldsymbol{q}_1,\boldsymbol{p}_1),  (\boldsymbol{q}_0,-\boldsymbol{p}_0) ) \pi((\boldsymbol{q}_1,\boldsymbol{p}_1))}
	  {q_h   ((\boldsymbol{q}_0,\boldsymbol{p}_0),  (\boldsymbol{q}_1,-\boldsymbol{p}_1) ) \pi((\boldsymbol{q}_0,\boldsymbol{p}_0))}   =  \\
 & \quad \exp   \left(  -\beta  \left(  \frac{1}{2} D_2 L_d ^T \boldsymbol{M}^{-1} D_2 L_d + U( \boldsymbol{q}_1 )
 						- \frac{1}{2} D_1 L_d ^T \boldsymbol{M}^{-1} D_1 L_d  - U( \boldsymbol{q}_0 )\right)  \right) \text{.}
\end{align*}
\end{proof}

The following lemma  is analogous to Lemma~\ref{MALAstagnationprobability}, and roughly speaking, quantifies how 
often rejections occur in MAGLA.

\begin{lemma}[Stagnation Probability of St\"{o}rmer-Verlet-Based MAGLA]
Consider the Metropolized St\"{o}rmer-Verlet based GLA \eqref{SVGLA}.
Assume \ref{sa} (C) and \ref{sa} (E) on the potential energy. For any integer $\ell \ge 1$, 
there exists $h_c>0$ and $K_{\ell} >0$, such that  for all positive $h<h_c$ and for all 
$\boldsymbol{x} =  (\boldsymbol{q}_0, \boldsymbol{p}_0) \in \mathbb{R}^n$,
\[
\E^{\boldsymbol{x}} \left\{   
( \alpha_h((\boldsymbol{q}_0, \boldsymbol{p}_0),(\boldsymbol{Q}_1^*, \boldsymbol{P}_1^*)) - 1 )^{2 \ell}  
\right\}  \le C(\boldsymbol{q}_0, \boldsymbol{p}_0) h^{6 \ell} \text{,}
\]
where
\begin{equation*} 
\begin{cases}
& \boldsymbol{Q}_{1}^* = \boldsymbol{q}_{0} + 
h \boldsymbol{M}^{-1} e^{-\gamma \boldsymbol{M}^{-1} h/2} \boldsymbol{p}_{0}  - \frac{h^2}{2} \boldsymbol{M}^{-1} \nabla U(  \boldsymbol{q}_{0} ) \\
& \qquad \qquad + h \sqrt{2 \beta^{-1} \gamma} \int_{t_k}^{t_k+h/2} \boldsymbol{M}^{-1} e^{- \gamma \boldsymbol{M}^{-1} (t_k+h/2-s) } d\boldsymbol{W}(s) \text{,} \\
& \boldsymbol{P}_{1}^* =  e^{-\gamma \boldsymbol{M}^{-1} h} \boldsymbol{p}_{0}
- \frac{h}{2} e^{- \gamma \boldsymbol{M}^{-1} h/2} \left( \nabla U( \boldsymbol{q}_{0}) + \nabla U( \boldsymbol{Q}_{1}^* ) \right) \\
& \qquad \qquad + \sqrt{2 \beta^{-1} \gamma} \int_{t_k}^{t_k+h} e^{- \gamma \boldsymbol{M}^{-1} (t_k + h -s ) } d \boldsymbol{W}(s) \text{.}
\end{cases}
\end{equation*}
\label{MAGLAstagnationprobability}
\end{lemma}

\begin{proof}
From Lemma~\ref{MAGLAacceptanceprobability} it is clear that,
\begin{align*}
& \E^{\boldsymbol{x}} \left\{   
( \alpha_h((\boldsymbol{q}_0, \boldsymbol{p}_0),(\boldsymbol{Q}_1^*, \boldsymbol{P}_1^*)) - 1 )^{2 \ell}  \right\}  = \\
& \qquad \int_{\mathbb{R}^{2n}}  \left( \exp(- \beta \Delta E(\boldsymbol{q}_0, \boldsymbol{q}_1) ) \wedge 1 - 1 \right)^{2 \ell} 
q_h( (\boldsymbol{q}_0, \boldsymbol{p}_0), (\boldsymbol{q}_1, \boldsymbol{p}_1)) d \boldsymbol{q}_1 d \boldsymbol{p}_1
\end{align*}
Substitute $q_h$ from \eqref{GLAdensity} into the above to obtain,
\begin{align*}
& \E^{\boldsymbol{x}} \left\{   
( \alpha_h((\boldsymbol{q}_0, \boldsymbol{p}_0),(\boldsymbol{Q}_1^*, \boldsymbol{P}_1^*)) - 1 )^{2 \ell}  \right\}  = \\
& \qquad \int_{\mathbb{R}^{2n}}   | \det(D_{12} L_d(\boldsymbol{q}_0, \boldsymbol{q}_1,h) )| 
\cdot \left( \exp(- \beta \Delta E(\boldsymbol{q}_0, \boldsymbol{q}_1) ) \wedge 1 - 1 \right)^{2 \ell}  \\
& \qquad \qquad \cdot o_{h/2}(\boldsymbol{p}_0, -D_1 L_d(\boldsymbol{q}_0, \boldsymbol{q}_1, h))  
\cdot o_{h/2} (D_2 L_d(\boldsymbol{q}_0, \boldsymbol{q}_1, h), \boldsymbol{p}_1)  d \boldsymbol{q}_1 d \boldsymbol{p}_1
\end{align*}
Integrate with respect to $\boldsymbol{p}_1$ to obtain the following simplified expression,
\begin{align*}
& \E^{\boldsymbol{x}} \left\{   
( \alpha_h((\boldsymbol{q}_0, \boldsymbol{p}_0),(\boldsymbol{Q}_1^*, \boldsymbol{P}_1^*)) - 1 )^{2 \ell}  \right\}  = \\
& \qquad \int_{\mathbb{R}^{n}}   | \det(D_{12} L_d ) | 
\cdot \left( \exp(- \beta \Delta E(\boldsymbol{q}_0, \boldsymbol{q}_1) ) \wedge 1 - 1 \right)^{2 \ell}  \\
& \qquad \qquad \cdot o_{h/2}(\boldsymbol{p}_0, -D_1 L_d(\boldsymbol{q}_0, \boldsymbol{q}_1, h))   d \boldsymbol{q}_1
\end{align*}
Let $R_h(\boldsymbol{q}_0) = \{ \boldsymbol{q}_1 \in \mathbb{R}^n : \Delta E(\boldsymbol{q}_0, \boldsymbol{q}_1) > 0 \}$.  Then,
\begin{align} \label{spinq1}
&   \E^{\boldsymbol{x}} \left\{   
( \alpha_h((\boldsymbol{q}_0, \boldsymbol{p}_0),(\boldsymbol{Q}_1^*, \boldsymbol{P}_1^*)) - 1 )^{2 \ell}  \right\}   \nonumber \\
& = \int_{R_h(\boldsymbol{q}_0)}   | \det(D_{12} L_d )| \cdot \left( e^{- \beta \Delta E(\boldsymbol{q}_0, \boldsymbol{q}_1)}  - 1 \right)^{2 \ell}  
\cdot o_{h/2}(\boldsymbol{p}_0, -D_1 L_d(\boldsymbol{q}_0, \boldsymbol{q}_1, h))   d \boldsymbol{q}_1\nonumber \\
& =   \frac{1}{( 2 \pi )^{n/2}  | \det( \boldsymbol{\Sigma}_{h/2}) | }   \int_{R_h(\boldsymbol{q}_0)}   | \det(D_{12} L_d )| \cdot 
\left( e^{- \beta \Delta E(\boldsymbol{q}_0, \boldsymbol{q}_1)} - 1 \right)^{2 \ell}  \nonumber \\
& \qquad \cdot e^{\left(-\frac{1}{2} 
\left(  -D_1 L_d(\boldsymbol{q}_0, \boldsymbol{q}_1, h) - e^{- \gamma \boldsymbol{M}^{-1} h/2} \boldsymbol{p}_0 \right)^T  \boldsymbol{\Sigma}_{h/2}^{-1} 
\left( -D_1 L_d(\boldsymbol{q}_0, \boldsymbol{q}_1, h) - e^{- \gamma \boldsymbol{M}^{-1} h/2} \boldsymbol{p}_0 \right) \right)} d \boldsymbol{q}_1  
\end{align}
Let $\boldsymbol{A}_{h/2}$ be the decomposition matrix arising from the Cholesky factorization 
of $\boldsymbol{\Sigma}_{h/2}$, i.e., $\boldsymbol{A}_{h/2} \boldsymbol{A}_{h/2}^T = \boldsymbol{\Sigma}_{h/2}$.
Introduce the map $\varphi: \mathbb{R}^n \to \mathbb{R}^n$, with $\boldsymbol{q}_1 = \varphi(\boldsymbol{\eta})$, 
and defined implicitly by the following relation:
\begin{equation} \label{q1toeta}
-D_1 L_d(\boldsymbol{q}_0, \boldsymbol{q}_1, h) = 
\exp(-\gamma \boldsymbol{M}^{-1} h/2) \boldsymbol{p}_0 + \boldsymbol{A}_{h/2} \boldsymbol{\eta} \text{.}
\end{equation}
Set $\Tilde{R}_h(\boldsymbol{q}_0) = \varphi^{-1} ( R_h(\boldsymbol{q}_0) )$.  A change of variables of \eqref{spinq1} 
under the map $\varphi$ yields,
\begin{align} \label{spineta2}
&  \E^{\boldsymbol{x}} \left\{   
( \alpha_h((\boldsymbol{q}_0, \boldsymbol{p}_0),(\boldsymbol{Q}_1^*, \boldsymbol{P}_1^*)) - 1 )^{2 \ell}  \right\}  
= \nonumber \\
& \qquad  ( 2 \pi )^{-n/2}  \int_{\Tilde{R}_h(\boldsymbol{q}_0)}   
\left( e^{- \beta \Delta E(\boldsymbol{q}_0, \varphi(\boldsymbol{\eta}))} - 1 \right)^{2 \ell}   
e^{-\frac{1}{2} | \boldsymbol{\eta} | } d \boldsymbol{\eta} 
\end{align}

Up to this point the argument has been for a general GLA.
Now, the proof is specialized to the St\"{o}rmer-Verlet based GLA.    
For the St\"{o}rmer-Verlet  discrete Lagrangian \eqref{VerletLd}:
\begin{align*}
D_1 L_d(\boldsymbol{q}_0, \boldsymbol{q}_1, h) &= 
	- \boldsymbol{M} \frac{\boldsymbol{q}_1 - \boldsymbol{q}_0}{h} - \frac{h}{2} \nabla U( \boldsymbol{q}_0 ) \text{,} \\
D_2 L_d(\boldsymbol{q}_0, \boldsymbol{q}_1, h) &=  
	\boldsymbol{M} \frac{\boldsymbol{q}_1 - \boldsymbol{q}_0}{h} - \frac{h}{2} \nabla U( \boldsymbol{q}_1 )  \text{.}
\end{align*}
Substitute these expressions into \eqref{DeltaE} gives:
\begin{align} \label{VerletDeltaE}
& \Delta E(\boldsymbol{q}_0, \boldsymbol{q}_1) =  U( \boldsymbol{q}_1 ) - U( \boldsymbol{q}_0 ) 
	- \frac{1}{2} \left\langle \nabla U( \boldsymbol{q}_1) 
	+ \nabla U( \boldsymbol{q}_0), \boldsymbol{q}_1 - \boldsymbol{q}_0 \right\rangle \nonumber \\
& \qquad + \frac{h^2}{8} \left( \nabla U( \boldsymbol{q}_1)^T \boldsymbol{M}^{-1} \nabla U( \boldsymbol{q}_1) 
	-  \nabla U( \boldsymbol{q}_0)^T \boldsymbol{M}^{-1} \nabla U( \boldsymbol{q}_0) \right)
\end{align}
Substitute the change of variables implicitly defined by \eqref{q1toeta} into \eqref{VerletDeltaE} and then
Taylor expand $\Delta E$ about $h=0$ to obtain,
\begin{align*}
& \Delta E\left(\boldsymbol{q}_0, \boldsymbol{q}_0 + 
h \boldsymbol{M}^{-1} e^{-\gamma \boldsymbol{M}^{-1} \frac{h}{2}} \boldsymbol{p}_0 
+ h \boldsymbol{M}^{-1} \boldsymbol{A}_{h/2} \boldsymbol{\eta} 
- \frac{h^2}{2} \boldsymbol{M}^{-1} \nabla U(\boldsymbol{q}_0) \right) =\\
& \qquad -\frac{h^3}{12} D^3 U( \boldsymbol{q}_0 ) \cdot \left( \boldsymbol{M}^{-1}  \boldsymbol{p}_0 \right)^3 
+ \frac{h^3}{4} D^2 U( \boldsymbol{q}_0) \cdot( \boldsymbol{M}^{-1} \boldsymbol{p}_0, \boldsymbol{M}^{-1} \nabla U( \boldsymbol{q}_0 ) ) \\
& \qquad + O(h^{7/2})
\end{align*}
As in Lemma~\ref{MALAstagnationprobability}, a standard application of Laplace's method together with
Assumption~\ref{sa} (E)  gives the desired result.
\end{proof}

Using Lemma~\ref{MAGLAstagnationprobability},  we are now in position to quantify the local accuracy of MAGLA.

\begin{lemma}[Local Accuracy of St\"{o}rmer-Verlet-Based MAGLA]
Assume \ref{sa} (C) and \ref{sa} (E) on the potential energy.  For all $h>0$ and 
$\boldsymbol{x} =  (\boldsymbol{q}_0, \boldsymbol{p}_0) \in \mathbb{R}^n$, there exists 
$C(\boldsymbol{q}_0, \boldsymbol{p}_0)>0$ such that $\mu(C) < \infty$ and
\begin{description}
\item[A)] the local mean-squared error of~MAGLA satisfies
\[
\left( \E^{\boldsymbol{x}} \left\{ \left| \boldsymbol{X}_1 -  \boldsymbol{Y}(h)  
\right|^2 \right\}  \right)^{1/2} 
\le C(\boldsymbol{q}_0, \boldsymbol{p}_0)  h^{3/2} \text{;}
\]
\item[B)] the local mean deviation of~MAGLA satisfies
\[
\left| \E^{\boldsymbol{x}} \left\{ \boldsymbol{X}_1 -  \boldsymbol{Y}(h)  \right\} \right| 
\le C(\boldsymbol{q}_0, \boldsymbol{p}_0) h^3 \text{.}
\]
\end{description}
\label{MAGLAlocalaccuracy}
\end{lemma}

\begin{Remark}
A detailed proof is omitted since similar steps are taken 
in the proof of Lemma \ref{MALAlocalaccuracy}.  
We remark that it follows from the definition of MAGLA \eqref{MAGLA},
\begin{align*}
 & \E^{\boldsymbol{x}} \left\{ \left| \boldsymbol{X}_1 -  \boldsymbol{Y}(h)  \right|^2 \right\}  \le  
 \E^{\boldsymbol{x}} \left\{ \left| \boldsymbol{X}_1^* -  \boldsymbol{Y}(h)  \right|^2 \right\} \\
&\qquad  +  \E^{\boldsymbol{x}} \left\{ \left| \varphi(\boldsymbol{q}_0, \boldsymbol{p}_0 ) -  \boldsymbol{Y}(h)  \right|^2 
 (1 - \alpha_h((\boldsymbol{q}_0, \boldsymbol{p}_0), (\boldsymbol{Q}_1^*, \boldsymbol{P}_1^*) ) ) \right\}
\end{align*}
This estimate reveals that the local mean-squared error of MAGLA is a sum of the local mean-squared
error of GLA (cf. Lemma \ref{GLAlocalaccuracy}) and the error that arises from a potential rejection weighted 
by the probability of rejection.   This error is $O(1)$ because when a rejection happens in \eqref{MAGLA}, 
the momentum is flipped.  However,  the probability of rejection is within the order of accuracy of the 
method according to Lemma \ref{MAGLAstagnationprobability}.  
\end{Remark}

\begin{Assumption}[Globally Lipschitz Continuous Process]
For $s \le t$,  let $\boldsymbol{Y}_{t,s}( \boldsymbol{x})$ denote 
the evolution operator of the solution to \eqref{SDE2}: 
with $\boldsymbol{Y}_{s,s}(\boldsymbol{x}) = \boldsymbol{x}$ and for  $r \le s \le t$ 
recall the Chapman-Kolmogorov identity 
$\boldsymbol{Y}_{t,s} \circ \boldsymbol{Y}_{s,r}(\boldsymbol{x}) = \boldsymbol{Y}_{t,r}(\boldsymbol{x})$.   
Set
\[
\boldsymbol{\Delta} =  
\boldsymbol{Y}_{s+h,s}(\boldsymbol{x}) - \boldsymbol{Y}_{s+h,s}(\boldsymbol{y}) - ( \boldsymbol{x} -\boldsymbol{y} ) \text{,}
\]
For every $K>0$ there exists a $h_c>0$, such that for all positive $h<h_c$, for all 
$\boldsymbol{x}, \boldsymbol{y} \in \mathbb{R}^n$, and for all $s\ge0$,
\begin{description}
\item[A)]
\[
\E \left\{ \left| \boldsymbol{Y}_{s+h,s}(\boldsymbol{x}) - \boldsymbol{Y}_{s+h,s}(\boldsymbol{y})  \right|^2 \right\} 
\le | \boldsymbol{x} - \boldsymbol{y} |^2  ( 1+  K h) \text{;}
\]
\item[B)]
\[
\E \left\{  \left| \boldsymbol{\Delta}  \right|^2 \right\} 
\le   K h^2 | \boldsymbol{x} - \boldsymbol{y} | \text{.}
\]
\end{description}
\label{sa2}
\end{Assumption}

As discussed MAGLA involves a momentum flip whenever a proposal move is rejected.  
Despite this MAGLA provides a pathwise approximant to \eqref{SDE2} as summarized by 
the following theorem.

\begin{thm23}[St\"{o}rmer-Verlet-Based MAGLA Strong Accuracy from Equilibrium]
Assume \ref{sa} (C) and \ref{sa} (E) on the potential energy, and \ref{sa2}  on \eqref{SDE2}.    
Then for every $T>0$,  there exists $h_c > 0$ and $C(T)>0$ such that for all
  $h<h_c$, for all $\boldsymbol{x} \in \mathbb{R}^{2n}$, and
  for all $t \in [0, T]$,
\[
\left( \E_{\mu}  \E^{\boldsymbol{x}} 
\left\{ \left| \boldsymbol{X}_{\lfloor t/ h \rfloor} - \boldsymbol{Y}(t) \right|^2 \right\} 
\right)^{1/2} \le C(T) h \text{.}
\]
\end{thm23}

\begin{Remark}
A detailed proof is omitted as it is similar to the proof of Theorem~\ref{MALAaccuracy}.
Recall, the main tools needed for that proof are local accuracy of MAGLA which follows
from Lemma~\ref{MAGLAlocalaccuracy}, and the invariance of $\mu$ under both
the transition kernel of the solution to \eqref{SDE2} and the Markov chain defined
by St\"{o}rmer-Verlet-based MAGLA.
\end{Remark}



\section{Conclusion} \label{sec:Conclusion}

This paper examined the pathwise accuracy of Metropolized integrators applied to overdamped
and inertial Langevin dynamics.   The paper primarily analyzed explicit discretizations of 
these ergodic SDEs.  For explicit discretizations if the drift is nonglobally Lipschitz, 
discretization errors often lead to stochastically unstable Markov chains.    A Metropolis-Hastings
method was shown to stochastically stabilize these discretizations.   The paper used this stochastic
stability to quantify the error induced by the Metropolis-Hastings methods in approximating the solution
to the SDE.

Although the paper restricted to specific discretizations of 
these SDEs as proposal moves, the strategy to prove the results is rather general.
The strategy can be used to prove pathwise accuracy of Metropolized versions of higher order accurate, adaptive
and multistep discretizations of overdamped or inertial Langevin dynamics (See, e.g., \cite{MiTr2004, VaCi2006}.).
In fact, this strategy applies to a larger class of SDEs, namely those which possess a transition density that
satisfies detailed balance-type condition.

Our strategy relies on two main ingredients. First, that the local error induced by the  
Monte-Carlo method depends upon the time-step size in an algebraic fashion.  
For example, often we found that the local mean-squared error of the Metropolized
Integrator can be decomposed as:
\begin{align*}
&\E^{\boldsymbol{x}}  \left\{    \left| \boldsymbol{X}_1 - \boldsymbol{Y}(h) \right|^2  \right\} = \\
 &\qquad  \underset{ \text{Accepted Proposal Move} }{\underbrace{ \E^{\boldsymbol{x}} \{ \left| \boldsymbol{X}^*_1 - \boldsymbol{Y}(h) \right|^2 \alpha_h(\boldsymbol{x},\boldsymbol{X}^*_1) \} }}
+ \underset{ \text{Rejected Proposal Move} }{\underbrace{ \E^{\boldsymbol{x}} \{  \left| \varphi(\boldsymbol{x})-\boldsymbol{Y}(h) \right|^2 (1- \alpha_h(\boldsymbol{x},\boldsymbol{X}^*_1))   \}  }} 
\end{align*}
where $\boldsymbol{X}_1$ is a single step of the Metropolis-adjusted discretization,
$\boldsymbol{Y}(h)$ is the exact solution of the SDE after a single step, 
$\boldsymbol{X}^*_1$ is the proposal move step (a single step of the unadjusted discretization), 
$ \alpha_h$ is the probability of accepting a proposal move, and 
$\varphi$ is a deterministic map applied if a proposal move is rejected.
The first term is bounded by $\E^{\boldsymbol{x}} \{ \left| \boldsymbol{X}^*_1 - \boldsymbol{Y}(h) \right|^2 \}$, 
i.e., the local order of accuracy of the discretization.
The second term is bounded by using the Cauchy-Schwarz inequality, together with 
bounds on $\E^{\boldsymbol{x}} \{  \left| \varphi(\boldsymbol{x})-\boldsymbol{Y}(h) \right|^4 \}^{1/2}$ and 
$\E^{\boldsymbol{x}} \{ (1- \alpha_h(\boldsymbol{x},\boldsymbol{X}^*_1))^2 \}^{1/2}$.
For MALA $\varphi$ was the identity, so that a rejected proposal
was nearby the solution. To be precise, 
$\E^{\boldsymbol{x}} \{  \left| \varphi(\boldsymbol{x})-\boldsymbol{Y}(h) \right|^4 \}^{1/2} \le O(h)$.
The likelihood of a rejection to occur turns out to be 
$\E^{\boldsymbol{x}} \{ (1- \alpha_h(\boldsymbol{x},\boldsymbol{X}^*_1))^2 \}^{1/2} \le O(h^{3/2})$.
Since the local mean-squared error of the unadjusted discretization was $O(h^{3/2})$, the local
mean-squared error of MALA was dominated by the error incurred by rejected proposal moves
which was $O(h^{5/4})$.

For MAGLA $\varphi$ was not the identity, but involved a momentum flip.   This momentum flip 
was imperative because the exact transition density of inertial Langevin dynamics does not 
satisfy detailed balance, but a modified detailed balance condition that is derived 
from composing the solution of inertial Langevin with a momentum flip.     
Consequently, a rejected proposal move for MAGLA loses local accuracy, i.e.,
$\E^{\boldsymbol{x}} \{  \left| \varphi(\boldsymbol{x})-\boldsymbol{Y}(h) \right|^4 \}^{1/2} \sim O(1)$.   
Nevertheless, the probability of rejection turns out to be within
the order of accuracy of the SV-based GLA, i.e., 
$\E^{\boldsymbol{x}} \{ (1- \alpha_h(\boldsymbol{x},\boldsymbol{X}^*_1))^2 \}^{1/2} \le O(h^3)$. 
Thus, we found the error incurred by a rejected proposal move to be of the same order as that
of an accepted proposal move.  Hence, the local mean-squared accuracy of MAGLA was 
$O(h^{3/2})$.

The second ingredient involves bounds on moments of the integrator that are 
uniform in the time-step size.  These bounds enabled the local error estimates
to be extended to global error estimates, and hence, pathwise convergence on finite time-intervals.  
These bounds are automatic if initial conditions are restricted to the 
known equilibrium distribution the Metropolis-adjusted integrator is designed to sample.
To relax this restriction on initial conditions, geometric ergodicity played a key role to 
obtain such bounds.  MALA is often ergodic, but because its proposal chain is 
often transient, is not geometrically ergodic.  This transience is due to a numerical
instability in forward Euler for large energy values.  By suitably truncating the drift 
at high energy values, one can ensure the proposal dynamics is not transient.   
MALA with truncated drift, or MALTA is known to often be geometrically ergodic on 
infinite time-intervals where MALA is not.  The paper used this property to 
prove MALTA is pathwise convergent on finite time-intervals without a restriction
on its initial condition.  We remark that MALTA can be generalized to explicit 
higher-order, multistep and adaptive integrators for overdamped and inertial
Langevin equations.

Finally, we emphasize that GLA can be extended to non-flat configuration spaces and 
multiple time-steps.  Indeed, its underlying variational integrator can 
incorporate holonomic constraints (via, e.g., Lagrange multipliers) \cite{WeMa1997} and 
multiple time steps to obtain so-called asynchronous variational integrators \cite{LeMaOrWe2003}. 
A natural choice for the variational integrator would be a SHAKE or RATTLE algorithm 
\cite{WeMa1997, MaWe2001}.  The Ornstein-Uhlenbeck equations 
\eqref{OrnsteinUhlenbeck} are defined on a flat vector space since the configuration 
is held fixed.  Moreover, a review of the proof of Lemma~\ref{MAGLAacceptanceprobability} 
shows that the probability transition density of GLA can be explicitly  characterized even if 
the configuration space is not flat.  By inspection, this density is absolutely 
continuous with respect to the standard volume measure on that non-flat space.   Consequently, 
GLA on manifolds can be used as proposal dynamics in a Metropolis-Hastings context, and 
MAGLA can be extended to non-flat configuration spaces to treat, e.g., inertial Langevin 
equations with holonomic constraints as in \cite{VaCi2006}.  Moreover, MAGLA 
on manifolds can readily classified as an ergodic Markov chain on that non-flat phase 
space.

\paragraph{Acknowledgements}

We thank Christof Sch\"{u}tte for stimulating discussions.  We thank 
Tony Lelievre, Sebastian Reich, and Gabriel Stoltz for suggestions on an 
earlier version of this paper.


\section{Proof of Lemmas~\ref{sdemomentbound},  \ref{regularityofsolutions} and \ref{ULAlocalaccuracy}}
\label{sec:prooflemma}

\paragraph{Proof of Lemma~\ref{sdemomentbound}}

\begin{description}
\item[Lemma~\ref{sdemomentbound} A)]
Let $G(\boldsymbol{x}) := U(\boldsymbol{x})^{\ell}$.  By the Taylor-Ito formula 
\[
 d G( \boldsymbol{Y}(t) ) = L \left\{ G (\boldsymbol{Y}(t) ) \right\} dt + \text{Martingale} \text{.}
\]
Assumption~\ref{sa}~(B) on the generator of the (SDE) implies
\[
\E^{\boldsymbol{x}} \left\{ G( \boldsymbol{Y}(t)  ) \right\} \le 
e^{- \delta_{\ell} t} G(\boldsymbol{x}) + \frac{M_{\ell}}{\delta_{\ell}} ( 1 - e^{-\delta_{\ell} t} ) \text{.}
\]
From this expression it is clear that  for all $t>0$ and for all $\boldsymbol{x} \in \mathbb{R}^n$
\[
\E^{\boldsymbol{x}} \left\{ G( \boldsymbol{Y}(t)  ) \right\} \le 
G(\boldsymbol{x}) + \frac{M_{\ell}}{\delta_{\ell}}  \text{.}
\]

\item[Lemma~\ref{sdemomentbound} B)]
Assumption~\ref{sa}~(A) implies 
\[
G(\boldsymbol{x}) \ge K^{\ell}  | \boldsymbol{x} |^{ \ell},~~\forall~\boldsymbol{x} \in \mathbb{R}^n  \text{.}
\]
Hence, the upper bound derived in the proof of Lemma~\ref{sdemomentbound} (A) 
on $\E^{\boldsymbol{x}} \left\{ G( \boldsymbol{Y}(t) ) \right\}$ implies an upper bound 
on $\E^{\boldsymbol{x}} \left\{  \left| \boldsymbol{Y}(t) \right|^{\ell} \right\} $.

\item[Lemma~\ref{sdemomentbound} C)]
From the solution to \eqref{SDE1} it is apparent
\[
\E^{\boldsymbol{x}} \left\{  | \boldsymbol{Y}(t) - \boldsymbol{x} |^{2 \ell}  \right\} 
= \E^{\boldsymbol{x}} \left\{  \left|  - \int_0^t \nabla U( \boldsymbol{Y}(s) ) ds + \sqrt{2 \beta^{-1} } \int_0^t d \boldsymbol{W} \right|^{2 \ell} \right\}
\]
The triangle inequality implies that,
\[
\E^{\boldsymbol{x}} \left\{  | \boldsymbol{Y}(t) - \boldsymbol{x} |^{2 \ell}  \right\}  
\le  \E^{\boldsymbol{x}} \left\{  \left( \left| \int_0^t \nabla U( \boldsymbol{Y}(s) ) ds \right|
						+ \sqrt{2 \beta^{-1} }  \left|   \int_0^t d \boldsymbol{W} \right|  \right)^{2 \ell} \right\}
\]
Since the function $g(x) = x^{2 \ell}$ is strictly convex,
\[
\E^{\boldsymbol{x}} \left\{  | \boldsymbol{Y}(t) - \boldsymbol{x} |^{2 \ell}  \right\}  
\le 2^{2 \ell - 1}  \E^{\boldsymbol{x}} \left\{  \left| \int_0^t \nabla U( \boldsymbol{Y}(s) ) ds \right|^{2 \ell} 
								+ (2 \beta^{-1})^{\ell}  \left|   \int_0^t d \boldsymbol{W} \right|^{2 \ell}  \right\}
\]
Recall, that the higher order central moments of a normal random variable 
$\xi \sim \mathcal{N}(0,t)$ are given by:
\[
\E \left\{ \xi^{2 \ell} \right\}  = \frac{(2 \ell)!}{2^{\ell} \ell!} t^{\ell} \text{.}
\]
Since the function $g(x) = x^{\ell}$ is convex for positive reals,
\[
\E^{\boldsymbol{x}} \left\{  \left|   \int_0^t d \boldsymbol{W} \right|^{2 \ell} \right\}  
\le n^{\ell} \E^{\boldsymbol{x}} \left\{  \xi^{2 \ell} \right\}  =  \frac{(2 \ell)!}{\ell!}  \left(\frac{n t}{2}\right)^{\ell}  \text{.}
\]
An application of Cauchy-Schwarz inequality yields,
\[
\E^{\boldsymbol{x}} \left\{ \left|  \int_0^t \nabla U( \boldsymbol{Y}(s) ) ds \right|^{2 \ell} \right\} 
\le t^{\ell } \E^{\boldsymbol{x}} \left\{ \left( \int_0^t \left|  \nabla U( \boldsymbol{Y}(s) ) \right|^2 ds \right)^{\ell} \right\} \text{.}
\]
According to Assumption~\ref{sa}~(E), there exists a constant $K>0$ such that
\[
\int_0^t \left|  \nabla U( \boldsymbol{Y}(s) ) \right|^2 ds  
\le  K \int_0^t (1+U(\boldsymbol{Y}(s) )^{2} ) ds  \text{.}
\]
An application of H\"{o}lder's inequality and Lemma~\ref{sdemomentbound} (A) 
imply there exists a possibly different constant $K>0$ such that
\[
\E^{\boldsymbol{x}} \left\{ \left|  \int_0^t \nabla U( \boldsymbol{Y}(s) ) ds \right|^{2 \ell} \right\}  
\le K t^{2 \ell } (1 + U(\boldsymbol{x})^{2 \ell} ) \text{.}
\]
Hence,
\[
\E^{\boldsymbol{x}} \left\{  | \boldsymbol{Y}(t) - \boldsymbol{x} |^{2 \ell}  \right\} 
\le  2^{2 \ell-1} \left(  t^{2 \ell} K (1+U(\boldsymbol{x})^{2 \ell} )  
				+  t^{\ell} \frac{(2 \ell)!}{\ell!}  \left(\frac{n}{\beta}\right)^{\ell}  \right) \text{.}
\]
\end{description}

\paragraph{Proof of Lemma~\ref{regularityofsolutions}.}

\begin{description}
\item[Lemma~\ref{regularityofsolutions} A)]
By the Taylor-Ito formula,
\begin{align*}
&\left| \boldsymbol{Y}_{s+h,s}(\boldsymbol{x}) - \boldsymbol{Y}_{s+h,s}(\boldsymbol{y})  \right|^2 = | \boldsymbol{x} - \boldsymbol{y} |^2 \\
&+ 2 \int_0^h \left\langle   \boldsymbol{Y}_{s+r,s}(\boldsymbol{x}) - \boldsymbol{Y}_{s+r,s}(\boldsymbol{y}),
\nabla U(\boldsymbol{Y}_{s+r,s}(\boldsymbol{y}))  - \nabla U(\boldsymbol{Y}_{s+r,s}(\boldsymbol{x})) \right\rangle dr
\end{align*}
According to Assumption~\ref{sa}~(D),
\begin{align*}
 \left| \boldsymbol{Y}_{s+h,s}(\boldsymbol{x}) - \boldsymbol{Y}_{s+h,s}(\boldsymbol{y})  \right|^2 
  \le | \boldsymbol{x} - \boldsymbol{y} |^2  
  + 2 K \int_0^h   \left|   \boldsymbol{Y}_{s+r,s}(\boldsymbol{x}) - \boldsymbol{Y}_{s+r,s}(\boldsymbol{y}) \right|^2  dr
\end{align*}
Gronwall's lemma implies that,
\begin{align*}
\left| \boldsymbol{Y}_{s+h,s}(\boldsymbol{x}) - \boldsymbol{Y}_{s+h,s}(\boldsymbol{y})  \right|^2 
\le | \boldsymbol{x} - \boldsymbol{y} |^2  \exp(2 h K)
\end{align*}

\item[Lemma~\ref{regularityofsolutions} B)]
To prove the second inequality, observe that:
\[
\left| \boldsymbol{\Delta}  \right|^2 = 
\left| \int_0^h \left( \nabla U(\boldsymbol{Y}_{s+r,s}(\boldsymbol{y}))  - \nabla U(\boldsymbol{Y}_{s+r,s}(\boldsymbol{x})) \right)  dr \right|^2
\]
An application of the Cauchy-Schwarz inequality implies,
\[
 \left| \boldsymbol{\Delta}  \right|^2 \le 
 h  \int_0^h \left|  \nabla U(\boldsymbol{Y}_{s+r,s}(\boldsymbol{y}))  - \nabla U(\boldsymbol{Y}_{s+r,s}(\boldsymbol{x})) \right|^2 dr 
\]
Assumption~\ref{sa}~(C) implies,
\[
 \left| \boldsymbol{\Delta}  \right|^2 \le 
 h  \int_0^h (U(\boldsymbol{Y}_{s+r,s}(\boldsymbol{y}))  +U(\boldsymbol{Y}_{s+r,s}(\boldsymbol{x})) ) 
 \left| \boldsymbol{Y}_{s+r,s}(\boldsymbol{y}) -\boldsymbol{Y}_{s+r,s}(\boldsymbol{x}) \right|^2 dr 
\]
Assumption~\ref{sa}~(A) and the triangle inequality imply there exists a constant $K>0$ such that
\[
 \left| \boldsymbol{\Delta}  \right|^2 \le 
 h  K \int_0^h (U(\boldsymbol{Y}_{s+r,s}(\boldsymbol{y}))^2  +U(\boldsymbol{Y}_{s+r,s}(\boldsymbol{x}))^2 ) 
 \left| \boldsymbol{Y}_{s+r,s}(\boldsymbol{y}) -\boldsymbol{Y}_{s+r,s}(\boldsymbol{x}) \right| dr 
\]
The Cauchy-Schwarz inequality implies that 
\begin{align*}
\E  \left\{ \left| \boldsymbol{\Delta}  \right|^2 \right\}  \le 
h  K \int_0^h & \left( \E\left\{ U(\boldsymbol{Y}_{s+r,s}(\boldsymbol{y}))^4 \right\}^{1/2}  
			    + \E\left\{  U(\boldsymbol{Y}_{s+r,s}(\boldsymbol{x}))^4 \right\}^{1/2}  \right)  \\
& \cdot \E \left\{ \left| \boldsymbol{Y}_{s+r,s}(\boldsymbol{y}) -\boldsymbol{Y}_{s+r,s}(\boldsymbol{x}) \right|^2  \right\}^{1/2} dr 
\end{align*}
Lemma~\ref{sdemomentbound} implies with a possibly different constant $K>0$
\begin{align*}
\E  \left\{ \left| \boldsymbol{\Delta}  \right|^2 \right\}  \le h  K  (1+ U(\boldsymbol{y})^2  + U(\boldsymbol{x})^2   ) 
\int_0^h \E \left\{ \left| \boldsymbol{Y}_{s+r,s}(\boldsymbol{y}) -\boldsymbol{Y}_{s+r,s}(\boldsymbol{x}) \right|^2  \right\}^{1/2} dr    \text{.}
\end{align*}
While the first part of this lemma implies, again with a possibly different constant $K>0$ that,
\begin{align*}
\E  \left\{ \left| \boldsymbol{\Delta}  \right|^2 \right\} 
 \le h^2  K  (1+ U(\boldsymbol{y})^2  +    U(\boldsymbol{x})^2   ) \left| \boldsymbol{x} - \mathbf{y} \right|  \text{.}
\end{align*}
\end{description}

\paragraph{Proof of Lemma~\ref{ULAlocalaccuracy}}

\begin{description} 
\item[Lemma~\ref{ULAlocalaccuracy} A)]
According to \eqref{SDE1} and \eqref{ULA}, the difference between the Euler-Maruyama 
discretization and the solution to the SDE after a single step takes the form:
\[
\Tilde{\boldsymbol{X}}_1 -  \boldsymbol{Y}(h)  
= - \int_0^h \left( \nabla U(\boldsymbol{x}) - \nabla U( \boldsymbol{Y}(s) \right) ds 
\]
The Cauchy-Schwarz inequality implies
\[
 \E^{\boldsymbol{x}} \left\{ \left| \Tilde{\boldsymbol{X}}_1 -  \boldsymbol{Y}(h)  \right|^2 \right\}  \le 
 h  \int_0^h \E^{\boldsymbol{x}} \left\{ \left| \nabla U(\boldsymbol{x}) - \nabla U( \boldsymbol{Y}(s) ) \right|^2 \right\} ds 
\]
Assumption~\ref{sa} (C) implies that there is a constant $K>0$ such that,
\begin{align*}
 \E^{\boldsymbol{x}} \left\{ \left| \Tilde{\boldsymbol{X}}_1  -  \boldsymbol{Y}(h)  \right|^2 \right\}  &\le 
 h K  \E^{\boldsymbol{x}} \left\{ \int_0^h ( U(\boldsymbol{x})^2 + U(\boldsymbol{Y}(s) )^2 )
  \left| \boldsymbol{x} - \boldsymbol{Y}(s) \right|^2 ds \right\}  \text{.}
\end{align*} 
A second application of Cauchy-Schwarz yields,
\begin{align*}
 \E^{\boldsymbol{x}} \left\{ \left|  \Tilde{\boldsymbol{X}}_1  -  \boldsymbol{Y}(h)  \right|^2 \right\}  &\le h K  \left(
 \int_0^h   U(\boldsymbol{x})^2 \E^{\boldsymbol{x}} \left\{ \left| \boldsymbol{x} - \boldsymbol{Y}(s) \right|^2 \right\} ds \right. \\
&+ \left.  \int_0^h \E^{\boldsymbol{x}} \left\{ U(\boldsymbol{Y}(s) )^4 \right\}^{1/2}  
\E^{\boldsymbol{x}} \left\{ \left| \boldsymbol{x} - \boldsymbol{Y}(s) \right|^4 \right\}^{1/2} ds 
 \right) \text{.}
\end{align*} 
Lemma~\ref{sdemomentbound} implies that there exists a constant $K>0$ such that
\begin{align*}
 \E^{\boldsymbol{x}} \left\{ \left| \Tilde{\boldsymbol{X}}_1 -  \boldsymbol{Y}(h)  \right|^2 \right\}   
 &\le K h^3 (1+ U(\boldsymbol{x})^4 )  \text{.}
\end{align*}

\item[Lemma~\ref{ULAlocalaccuracy} B)]
The Ito-Taylor formula implies that,
\[
\left| \E^{\boldsymbol{x}} \left\{ \Tilde{\boldsymbol{X}}_1 -  \boldsymbol{Y}(h) \right\}  \right|^2 =  
\left| \E^{\boldsymbol{x}}  \left\{ \int_0^h \int_0^s L \left\{  \nabla U( \boldsymbol{Y}(r) ) \right\} dr ds \right\} \right|^2
\]
Applying the Cauchy-Schwarz inequality twice gives,
\begin{align*}
\left| \E^{\boldsymbol{x}} \left\{ \Tilde{\boldsymbol{X}}_1 -  \boldsymbol{Y}(h) \right\}  \right|^2 
&\le  h \int_0^h  \left| \E^{\boldsymbol{x}}  \left\{ \int_0^s L \left\{  \nabla U( \boldsymbol{Y}(r) ) \right\} dr  \right\} \right|^2 ds  \\
&\le  h \int_0^h s  \int_0^s \left| \E^{\boldsymbol{x}}  \left\{ L \left\{  \nabla U( \boldsymbol{Y}(r) ) \right\}  \right\} \right|^2 dr ds 
\end{align*}
Jensen's inequality implies that,
\begin{align*}
\left| \E^{\boldsymbol{x}} \left\{ \Tilde{\boldsymbol{X}}_1 -  \boldsymbol{Y}(h) \right\}  \right|^2 
&\le  h \int_0^h s  \int_0^s  \E^{\boldsymbol{x}}  \left\{  \left| L \left\{  \nabla U( \boldsymbol{Y}(r) ) \right\}  \right|^2 \right\} dr ds 
\end{align*}
Assumption~\ref{sa} (E) implies the existence of a $K>0$ such that
\begin{align*}
\left| \E^{\boldsymbol{x}} \left\{ \Tilde{\boldsymbol{X}}_1 -  \boldsymbol{Y}(h) \right\}  \right|^2 
&\le  h \int_0^h s  \int_0^s     K ( 1+ \E^{\boldsymbol{x}} \left\{ U( \boldsymbol{Y}(r))^4  \right\} )   dr ds 
\end{align*}
An application of Lemma~\ref{sdemomentbound}  gives the required result.
\end{description}

\newpage

\bibliographystyle{amsplain}
\bibliography{nawaf}

 \end{document}